\documentclass[11pt]{article}

\usepackage[margin=1in]{geometry}
\usepackage{amsmath,amssymb,amsthm,mathtools,color}
\usepackage{enumitem}
\usepackage{hyperref}

\usepackage{tikz}
\usetikzlibrary{arrows.meta, positioning}

\newcommand{\E}{\mathbb{E}}
\newcommand{\Var}{\mathrm{Var}}
\newcommand{\Cov}{\mathrm{Cov}}

\newcommand{\Z}{\mathbb{Z}}
\newcommand{\R}{\mathbb{R}}
\newcommand{\N}{\mathbb{N}}

\newcommand{\Azero}{\textup{(\ref{ass:standing:stationarity})}}
\newcommand{\Aone}{\textup{(\ref{ass:standing:moments})}}
\newcommand{\Atwo}{\textup{(\ref{ass:standing:mean_birth_rates})}}
\newcommand{\Athree}{\textup{(\ref{ass:standing:linear_growth})}}

\theoremstyle{plain}
\newtheorem{theorem}{Theorem}
\newtheorem{lemma}{Lemma}
\newtheorem{proposition}{Proposition}
\newtheorem{corollary}{Corollary}

\theoremstyle{definition}
\newtheorem{definition}{Definition}

\newtheorem{example}{Example}
\newtheorem{assumption}{Assumption}[section]
\newtheorem{condition}{Condition}

\title{Noise Tradeoffs, Stationary Information Flow, and Structural Balance in Unit-Birth Networks}
\author{David F.~Anderson\thanks{Department of Mathematics, University of Wisconsin--Madison, Madison, WI, USA. Email: \texttt{anderson@math.wisc.edu}.}}
\date{\today}

\begin{document}
\maketitle

\begin{abstract}
In 2019, Paulsson and collaborators conjectured that stochastic biochemical
control networks exhibit fundamental limits on how much intrinsic noise can
be simultaneously suppressed across multiple components. While supported by
simulations and partial analytic results, a general proof remained
unavailable. Recently, Ripsman, Kell, and Hilfinger proposed a formal proof strategy
for unit-birth models based on a stationary information-theoretic decomposition. The purpose of this paper is
to provide the corresponding rigorous mathematical justification.

We consider continuous-time Markov chains on $\Z^N_{\ge 0}$
in which each component is degraded linearly and is produced in unit births
at a state-dependent rate that may depend on the other components, but not
directly on itself. Noise in component~$i$ is measured by the Fano factor
$F_{X_i}$, the ratio of stationary variance to stationary mean, with the
Poisson value~$1$ as the natural baseline. Our first contribution is to isolate explicit hypotheses on moments, mean birth
rates, and growth of the total jump rate under which the formal
information-flow identities can be rigorously justified on the infinite state
space.
 Under these hypotheses, and following
the proof outline of Ripsman, Kell, and Hilfinger, we prove that the resulting tradeoff
\[
\sum_{i=1}^N \frac{F_{X_i}-1}{\tau_i}\ge 0
\]
holds for \textit{any} network topology and \textit{any} allowable choice of rate functions.
Thus uniform sub-Poissonian noise suppression is impossible: if $F_{X_i} < 1$ for some $i$, then $F_{X_j} > 1$ for some $j \ne i$.

Our second contribution is to make these hypotheses checkable: a uniform
positive lower bound on the birth rates together with at-most-linear total
growth dominated by the weakest degradation rate suffices, which we prove with the use of a
Foster--Lyapunov function.

Our third contribution is a structural strengthening. Under a signed
monotonicity condition on the rate functions, satisfied in particular by
structurally balanced signed interaction networks, we prove that the stationary
distribution is associated with respect to the corresponding signed partial
order. This upgrades the global tradeoff to the termwise bound
$F_{X_i}\ge 1$ for every~$i$. Consequently, within the signed-monotone subclass, sub-Poissonian noise
requires a frustrated interaction topology.
\end{abstract}

\section{Introduction}

Stochastic reaction networks are a standard modeling framework for biochemical
systems in which molecular species are present in low copy numbers
\cite{gillespie1977exact, mcadams1997stochastic, thattai2001intrinsic,
elowitz2002stochastic, paulsson2004summing, raj2008nature}. In such systems,
individual production and degradation events can generate substantial
fluctuations in molecular abundances, and these fluctuations may propagate
through regulatory networks. Such fluctuations can affect cellular switching,
signaling, timing, and the reliability of downstream responses
\cite{mcadams1997stochastic,paulsson2004summing,raj2008nature}.

To discuss whether regulatory networks can suppress these fluctuations, one
first needs a baseline. For a single species produced at a constant rate and
degraded linearly, the stationary distribution is Poisson \cite{anderson2015stochastic}; in particular, its variance equals its mean. Thus the Fano factor
\[
F_X := \frac{\Var(X)}{\E[X]}
\]
equals one.  This Poisson value is the natural noise baseline for the models considered
here. Following the terminology used in the noise-suppression literature, we
will say that a component is sub-Poissonian if $F_X<1$ and super-Poissonian if
$F_X>1$; Yan, Hilfinger, Vinnicombe, and Paulsson, for example, refer to
components below the Poisson baseline as exhibiting sub-Poisson fluctuations
\cite{yan2019kinetic}.

Biological networks, however, are rarely collections of independent
birth-death processes. One species may regulate another through transcriptional
regulation, signaling pathways, sequestration, enzymatic activity, or other
interaction mechanisms
\cite{mcadams1997stochastic, shen2002network, tyson2003sniffers,
purvis2013encoding, buchler2009protein}. It is therefore natural to ask whether
a sufficiently well-designed regulatory network could suppress fluctuations in
many components simultaneously. For example, if $X_1$ happens to fluctuate
upward, one could imagine that this increases the production of a second
component $X_2$, which in turn represses the production of $X_1$ and thereby
pushes it back toward its mean. More complicated networks could use many such
interactions, and one might imagine that the collective effect is to reduce the
noise level of every component.

A natural question is whether such uniform suppression is possible. That is,
given the freedom to choose any topology and any nonlinear rate functions, can
one design an interaction network whose components all have Fano factors below
the Poisson baseline? Paulsson and collaborators conjectured that, for a broad
class of models, the answer is no \cite{yan2019kinetic}. Their conjecture
allows for various production mechanisms, including bursty production, while
restricting to linear degradation. The present paper treats the unit-birth
case, in which each production event increments a single component by one. In
biochemical terms, the variables $X_i$ may be interpreted as protein copy
numbers or molecular species abundances: each component is degraded linearly,
while its production rate may be regulated arbitrarily by the other
components, but not directly by itself. Thus the model excludes direct
self-feedback in production, while still allowing arbitrary indirect feedback
through the rest of the network.

The conjectured principle is striking because of its generality: it is not a
statement about a particular topology, a particular choice of rate functions,
or a perturbative regime. Rather, it says that no matter how the components
regulate one another's production rates, the network cannot push every
component below the Poisson noise baseline at stationarity. Equivalently, if
one component has $F_{X_i}<1$, then at least one other component must have
$F_{X_j}>1$. See Figure~\ref{fig:inaccessible_regions}(a) for a visual
representation in the two-component case.

\begin{figure}[h]
\centering
\begin{tikzpicture}[scale=2.0]

% =====================================
% Panel (a): general unit-birth network
% =====================================
\begin{scope}[xshift=0cm]

\fill[red!15] (0,0) rectangle (1,1);

\draw[->, thick] (0,0) -- (2.7,0) node[right] {$F_{X_1}$};
\draw[->, thick] (0,0) -- (0,2.7) node[above] {$F_{X_2}$};

\draw[thick, red] (1,0) -- (1,1) -- (0,1);

\draw[dashed, gray] (1,1) -- (1,2.5);
\draw[dashed, gray] (1,1) -- (2.5,1);

\fill[black] (1,1) circle (1pt);
\node[above right, font=\scriptsize] at (1,1) {$(1,1)$};

\node[below, font=\scriptsize] at (1,0) {$1$};
\node[left, font=\scriptsize] at (0,1) {$1$};

\node[align=center, font=\scriptsize] at (0.55,0.5) {Impossibility\\region};

\node[font=\small\bfseries] at (1.35,-0.4) {(a)};

\end{scope}

% =====================================
% Panel (b): structurally balanced network
% =====================================
\begin{scope}[xshift=4cm]

\fill[red!15] (0,0) -- (2.5,0) -- (2.5,1) -- (1,1) -- (1,2.5) -- (0,2.5) -- cycle;

\draw[->, thick] (0,0) -- (2.7,0) node[right] {$F_{X_1}$};
\draw[->, thick] (0,0) -- (0,2.7) node[above] {$F_{X_2}$};

\draw[thick, red] (1,2.5) -- (1,1) -- (2.5,1);

\fill[black] (1,1) circle (1pt);
\node[above right, font=\scriptsize] at (1,1) {$(1,1)$};

\node[below, font=\scriptsize] at (1,0) {$1$};
\node[left, font=\scriptsize] at (0,1) {$1$};

\node[align=center, font=\scriptsize] at (0.55,0.5) {Impossibility\\region};

\node[font=\small\bfseries] at (1.35,-0.4) {(b)};

\end{scope}

\end{tikzpicture}

\caption{Impossibility regions in $(F_{X_1}, F_{X_2})$ space for two-component
unit-birth networks.
(a) Illustrates the impossibility of simultaneous sub-Poissonian noise
suppression: for any network topology and any allowable rate functions, the
two Fano factors cannot both lie below the Poisson baseline $1$
(Theorem~\ref{thm:main}). The theorem in fact establishes the stronger
weighted tradeoff $(F_{X_1}-1)/\tau_1+(F_{X_2}-1)/\tau_2\ge 0$, which forbids
a half-plane bounded by a slanted line through $(1,1)$ whose slope depends on
the timescale ratio $\tau_1/\tau_2$; the unit square shown here is the
topology- and timescale-independent piece that holds for any choice of $\tau_i$.
Hilfinger et al.~\cite{RKH2025} plot the equivalent quantity
$\mathrm{CV}_i/\sqrt{1/\langle X_i\rangle}=\sqrt{F_{X_i}}$; the corresponding
regions are the same.
(b) Under the additional signed monotonicity hypothesis
(Assumption~\ref{ass:signed_monotonicity}), each Fano factor must individually
satisfy $F_{X_i}\ge 1$ (Theorem~\ref{thm:local_bound_signed_monotonicity}); the
impossibility region enlarges accordingly.}
\label{fig:inaccessible_regions}
\end{figure}
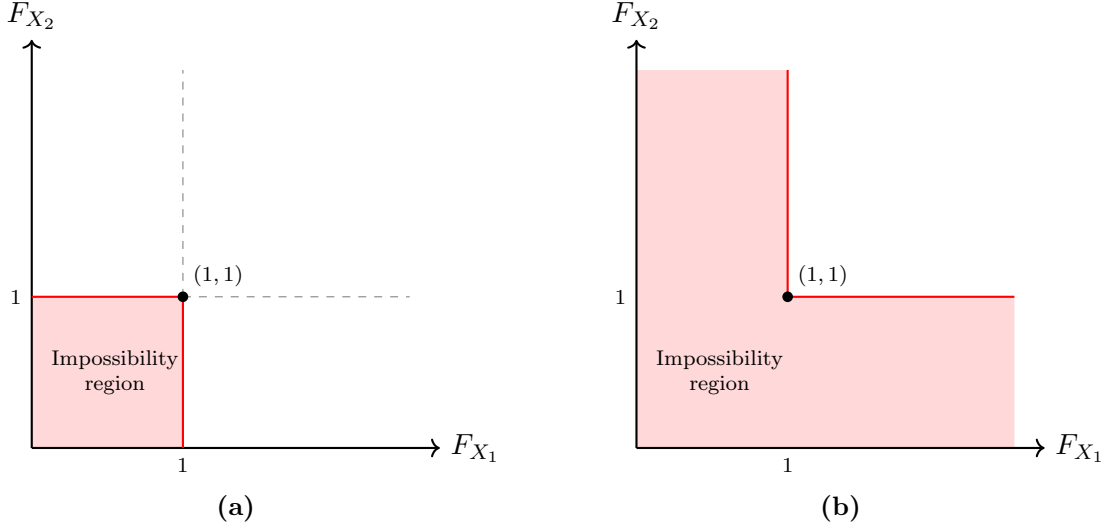

  Until recently, no general proof strategy existed for this conjecture, even for restricted settings.
A recent formal proof strategy of Ripsman, Kell, and Hilfinger
\cite{RKH2025} identifies the correct mechanism behind this tradeoff. Their
argument decomposes stationary information-theoretic quantities, including
mutual information rates, into directional information flows associated with
the different components. At a formal level, these flows satisfy a global
cancellation identity, and this cancellation combines with componentwise
information inequalities to yield the desired noise tradeoff.

One purpose of the present paper is to provide a rigorous mathematical
foundation for this argument in the unit-birth model. The main difficulty is
that the natural state space is infinite. Consequently, the entropy-like
quantities appearing in the information-flow decomposition involve logarithms
of stationary probabilities and conditional stationary probabilities, and the
formal manipulations require interchanging infinite sums, applying generators
to unbounded log-potentials, and justifying stationary generator identities. For example, the log-potential $g(x) = -\log\pi(x)$ is unbounded, and the stationary distribution $\pi$ enters our arguments only through the standing hypotheses rather than via any explicit formula, so the integrability of expressions involving $g$ requires careful verification.  %These are precisely the points where the formal information-flow calculation requires analytical verification. 

Our first contribution, therefore, is to isolate explicit hypotheses under which the
information-flow argument is valid. 
The most important moment assumption is a
finite second-moment condition, supplemented by natural regularity conditions
such as irreducible positive recurrence, full support of the stationary
distribution, finite positive mean birth rates, and mild total-rate growth.
Under these hypotheses we rigorously justify the formal identities used by
Ripsman, Kell, and Hilfinger. In particular, we prove the exact
Fano--covariance identity
\[
F_{X_i}
=
1+\frac{\Cov_\pi(f_i(X_{-i}),X_i)}{\E_\pi[f_i(X_{-i})]},
\]
where $f_i$ is the production rate of component $i$ and $X_{-i}$ denotes the vector $X$ with its $i$th coordinate removed, derive the componentwise information inequality, establish the stationary
global cancellation identity, and conclude the tradeoff
\[
\sum_{i=1}^N \frac{F_{X_i}-1}{\tau_i}\ge 0.
\]
This is the first impossibility result of the paper: within this framework,
uniform sub-Poissonian noise suppression is impossible. If one component has
sub-Poissonian noise, then the deficit must be balanced by super-Poissonian
noise in at least one other component.

Our second contribution is to make the hypotheses checkable. We give concrete
sufficient conditions on the birth-rate functions that imply the standing
assumptions used in the proof. In particular, a uniform positive lower bound on
the birth rates, together with at-most-linear growth of the total birth rate
with linear coefficient dominated by the weakest degradation rate, yields
nonexplosion, positive recurrence, finite moments, and the integrability needed
for the information-flow identities. This separates the abstract stationary
identities from the model-level conditions that guarantee they apply.

Our third contribution is a new structural theorem identifying a broad
topological regime in which the global tradeoff strengthens to a termwise bound.
Specifically, we show that for networks satisfying a signed monotonicity condition, the global inequality
upgrades to the following:
\[
F_{X_i}\ge 1
\qquad\text{for every }i.
\]
A broad graph-theoretic class satisfying this condition is given by
structurally balanced signed interaction networks, namely networks whose
signed interaction graph contains no cycle with an odd number of negative
edges. See Figure~\ref{fig:structural_balance_examples} for examples. In this case, a suitable signed order makes the stochastic dynamics
order-preserving. We prove that the stationary distribution is associated with
respect to this order, in the sense of Esary et al.~and of Lindqvist
\cite{esary1967association,lindqvist1988association},   and use the exact Fano--covariance identity to show that
each covariance
\[
\Cov_\pi(f_i(X_{-i}),X_i)
\]
is nonnegative. See Figure~\ref{fig:inaccessible_regions}(b) for the corresponding impossibility region.

\begin{figure}
\centering
\begin{tikzpicture}[
    node distance={2.3cm},
    main/.style = {draw, circle, thick},
    activation/.style = {-{Latex[length=2.5mm]}, thick, shorten >=2pt},
    inhibition/.style = {-{Bar[width=3mm]}, thick, shorten >=4pt}
]

% Panel (a): 3-node structurally balanced network
\begin{scope}[xshift=0cm]
\node[main] (1a) {$X_1$};
\node[main] (2a) [right of=1a] {$X_2$};
\node[main] (3a) [below of=2a] {$X_3$};

\draw[inhibition, red] (1a) -- node[midway, above] {$-$} (2a);
\draw[activation, blue] ([xshift=-4pt]2a.south) -- node[midway, left] {$+$} ([xshift=-4pt]3a.north);
\draw[activation, blue] ([xshift=4pt]3a.north) -- node[midway, right] {$+$} ([xshift=4pt]2a.south);
\draw[inhibition, red] (3a) -- node[midway, below left] {$-$} (1a);

\node[font=\small\bfseries] at (1.15, -3.2) {(a)};
\end{scope}

% Panel (b): repressilator
\begin{scope}[xshift=5cm]
\node[main] (1b) {$X_1$};
\node[main] (2b) [right of=1b] {$X_2$};
\node[main] (3b) [below of=2b] {$X_3$};

\draw[inhibition, red] (1b) -- node[midway, above] {$-$} (2b);
\draw[inhibition, red] (2b) -- node[midway, right] {$-$} (3b);
\draw[inhibition, red] (3b) -- node[midway, below left] {$-$} (1b);

\node[font=\small\bfseries] at (1.15, -3.2) {(b)};
\end{scope}

\end{tikzpicture}

\caption{Examples of signed interaction graphs for $N=3$ component networks.
Activating edges are drawn in blue with an arrowhead and labeled $+$;
inhibiting edges are drawn in red with a T-bar terminator and labeled $-$.
(a) A structurally balanced network: every cycle contains an even number of negative edges and Theorem~\ref{thm:local_bound_signed_monotonicity} gives the
termwise bound $F_{X_i}\ge 1$ for every $i$.
(b) The repressilator of Elowitz and Leibler~\cite{elowitz2000synthetic}, in
which each species represses the next. The signed interaction graph is a
directed $3$-cycle with three negative edges; it is not structurally balanced,
and the signed monotonicity hypothesis fails.}
\label{fig:structural_balance_examples}
\end{figure}
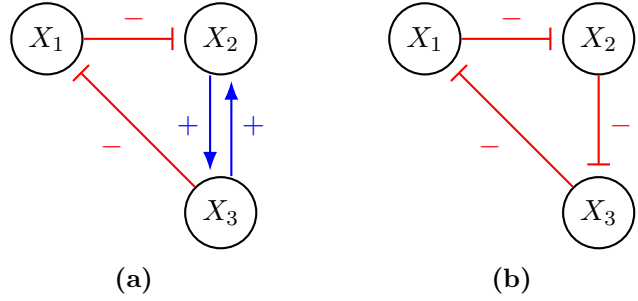

This gives a \textit{second} impossibility result, strictly stronger than the first: in a structurally balanced
network, even componentwise sub-Poissonian noise suppression is impossible.  
Thus, within the
class of models whose interactions admit a global sign description, any attempt
to construct a component with $F_{X_i}<1$ must use a frustrated interaction
structure, meaning a signed interaction graph that is not structurally
balanced. In this sense, frustration is not merely one way to obtain noise
suppression; it is a necessary mechanism for sub-Poissonian noise in this
class. This places a hard structural limit on what regulatory architecture can
achieve.

The remainder of the paper is organized as follows.
Section~\ref{sec:model_hypotheses_statement} introduces the unit-birth model,
the standing hypotheses, and the main tradeoff theorem, and then gives
checkable sufficient conditions under which those hypotheses hold.
Section~\ref{sec:technical_lemmas} proves the technical generator and
stationary identities needed for the information-flow calculation. Some of
these are general countable-state Markov chain tools, while the final subsection
applies them to the unit-birth logarithmic potentials, information flows, and
marginal flux balance. Section~\ref{sec:proof}
provides the proof of the main result. Specifically, it establishes the
componentwise information inequalities, the cancellation of stationary
information flows, and the resulting global Fano-factor tradeoff.
Section~\ref{sec:structural_balance} proves the termwise bound under signed
monotonicity and relates this condition to structurally balanced interaction
graphs and monotone-systems theory.

\section{The unit-birth model: hypotheses, main theorem, and sufficient conditions}
\label{sec:model_hypotheses_statement}
In this section we introduce the model, fix notation, and state the standing
hypotheses under which the stationary identities below are justified. We then
state the main ``noise tradeoff'' theorem for unit-birth models, and give a concrete sufficient condition showing
that the standing hypotheses hold for a broad class of unit-birth models.

\subsection{The unit-birth model}

Before defining the model, we fix a few pieces of notation. For a
continuous-time Markov chain on a countable state space $S$, we write $q(x,y)$
for the transition rate from $x$ to $y$. For $x\in\R^N$, we write $x_i$ for
the $i$th coordinate and $x_{-i}$ for the vector obtained by deleting the
$i$th coordinate.

We will require that any model found in this paper satisfies the following regularity conditions that can be found as Condition 1 in \cite{anderson2018non}.

\begin{condition}
\label{condition:regularity_minimal}
We require the following.  
\begin{itemize}
\item For $x,y\in S$ with $x\ne y$, we have $0\le q(x,y)< \infty$.
\item For each $y\in S$, we have $\sum_{x\in S\setminus\{y\}} q(y,x) < \infty$.
\item For each $x \in S$, we have $\sup_{y\in S \setminus\{x\}} q(y,x) < \infty$.
\end{itemize}
\end{condition}

We now turn to our model of interest.
Fix an integer $N\ge 2$.  Let $X$ be a continuous-time Markov chain on
\[
S:=\Z_{\ge 0}^N
\]
with generator $L$ acting on functions $\varphi:S\to\R$ by
\begin{equation}\label{eq:generator}
(L\varphi)(x)
:=\sum_{i=1}^N f_i(x_{-i})\big(\varphi(x+e_i)-\varphi(x)\big)
+\sum_{i=1}^N \frac{x_i}{\tau_i}\big(\varphi(x-e_i)-\varphi(x)\big),
\end{equation}
where $\tau_i>0$, and $e_i$ is the $i$th coordinate vector, and with the convention that the death summand is omitted when $x_i=0$.
Thus coordinate $i$ undergoes births at rate $f_i(x_{-i})$ (depending only on the other coordinates)
and deaths at rate $x_i/\tau_i$, with $f_i: \Z_{\ge 0}^{N-1} \to \mathbb{R}_{\ge 0}$.  Because the model has only single-coordinate birth and death transitions,
Condition~\ref{condition:regularity_minimal} is automatic whenever the rates
$f_i(v)$ are finite.
For later use, we write
\[
L=\sum_{i=1}^N L_i,
\]
where
\begin{equation}\label{eq:Li_def}
(L_i\varphi)(x)
:=
f_i(x_{-i})\big(\varphi(x+e_i)-\varphi(x)\big)
+
\frac{x_i}{\tau_i}\big(\varphi(x-e_i)-\varphi(x)\big),
\end{equation}
with the death term omitted when $x_i=0$.

For a generic Markov chain with transition rates  $q(x,y)$, we will denote the total jump rate out of state $x$   as
\begin{equation}\label{eq:Lambda_def}
\Lambda(x):=\sum_{y\neq x} q(x,y).
\end{equation}
Hence, for our model defined above, we have
\begin{equation*}
\Lambda(x)=\sum_{i=1}^N f_i(x_{-i})+\sum_{i=1}^N \frac{x_i}{\tau_i}.
\end{equation*}
Also define
\begin{equation}
\label{eq:linear_s}
s(x):=\|x\|_1=\sum_{j=1}^N x_j.
\end{equation}

Let $\pi$ denote a stationary distribution for $X$ (when it exists), and write $\E_\pi[\cdot]$ for expectation under $\pi$.
For a fixed $i\in \{1,\dots, N\}$,  we will often have need to separate out the $i$th component.  We therefore perform the following slight abuse of notation.  For a given $x\in \Z^N_{\ge 0}$ with $x_i = n$ and $x_{-i} = v$, we will often write $\pi(x) = \pi(n,v)$.  The meaning should always be clear in context.

Let $\pi_i$ denote the marginal law of $X_i$ and, for each $n$ with $\pi_i(n)>0$, let $\pi_n^i$ denote the conditional law of $X_{-i}$ given $X_i=n$:
\begin{align}
\label{eq:pi_marginals}
\pi_i(n):=\sum_{v\in \Z^{N-1}_{\ge 0} } \pi(n,v),
\qquad
\pi_n^i(v):=P_\pi(X_{-i}=v\mid X_i=n)=\frac{\pi(n,v)}{\pi_i(n)}.
\end{align}
We will also use the logarithmic potentials
\begin{equation}
\label{eq:log_potentials}
g(n,v):=-\log\pi(n,v),
\qquad
g_i(n,v):=-\log\pi_i(n),
\qquad
h_i(n,v):=-\log\pi_n^i(v).
\end{equation}
Under Assumption~\ref{ass:standing}, the stationary distribution has full
support, and hence these quantities are finite for every $(n,v)\in S$.
Moreover, since
\[
\pi(n,v)=\pi_i(n)\pi_n^i(v),
\]
we have, for each $i \in \{1,\dots,N\}$, the decomposition
\begin{equation}\label{eq:g_decomp}
g(n,v)=g_i(n,v)+h_i(n,v).
\end{equation}
Define
\begin{align}
\label{eq:mean_terms}
m_i(n):=\E_{\pi_n^i}[f_i(X_{-i})],
\qquad
\mu_i:=\E_\pi[f_i(X_{-i})]=\E_\pi[m_i(X_i)].
\end{align}
Finally, define the Fano factor
\[
F_{X_i}:=\frac{\Var_\pi(X_i)}{\E_\pi[X_i]}.
\]

\subsection{Standing hypotheses}

The following standing hypotheses will be used throughout the paper. They give
a convenient rigorous framework in which the stationary identities and 
information-flow decompositions developed by Hilfinger et al.~\cite{RKH2025} can be proved. The hypotheses are tailored to the unit-birth setting, but several of the
technical lemmas below are stated in a form that applies more generally to
countable-state continuous-time Markov chains satisfying the corresponding
regularity and integrability conditions.

\begin{assumption}[Standing hypotheses]\label{ass:standing}
Assuming Condition \ref{condition:regularity_minimal} is satisfied, we also assume the following: 
\begin{enumerate}[label=\textup{(A\arabic*)}, ref=A\arabic*, start=0, leftmargin=3em]
\item \label{ass:standing:stationarity}
\emph{Irreducibility and stationarity.}
The CTMC $X$ is irreducible on $S$ and admits a stationary distribution $\pi$.

\item \label{ass:standing:moments}
\emph{Finite second moments.}
$\E_\pi[X_i^2]<\infty$ for each $i$. Equivalently, $\E_\pi[s(X)^2]<\infty$.

\item \label{ass:standing:mean_birth_rates}
\emph{Finite, positive mean birth rates.}
For each $i$,
\[
\mu_i=\E_\pi[f_i(X_{-i})]\in(0,\infty).
\]

\item \label{ass:standing:linear_growth}
\emph{At-most-linear total rate growth.}
There exists a constant $C_\Lambda<\infty$ such that
\[
\Lambda(x)\le C_\Lambda(1+s(x))
\qquad\forall x\in S.
\]
\end{enumerate}
\end{assumption}

Two simple consequences of Assumption~\ref{ass:standing} will be used
repeatedly below. First, since the chain is irreducible and admits a stationary
distribution under \Azero, the stationary distribution has full support:
$\pi(x)>0$ for every $x\in S$. Thus the conditional laws and logarithmic
potentials appearing below are well defined. Second, although nonexplosion is
not included as a standing hypothesis, it follows from \Azero, \Aone, and
\Athree by Lemma~\ref{lem:nonexplosive_from_moments}.

%----------------------------
\subsection{Statement of the main theorem}
\label{sec:statement_main_theorem}
%----------------------------

We now state the main theorem. Recall the marginal law $\pi_i$, the
conditional laws $\pi_n^i$, the mean birth rates $\mu_i$, and the Fano factors
$F_{X_i}$ defined above in \eqref{eq:pi_marginals} and \eqref{eq:mean_terms}. By the full-support consequence of \Azero, discussed
above, $\pi_i(n)>0$ for every $n\ge0$, so the conditional laws $\pi_n^i$ are
well defined for every $n\ge0$. When focusing on coordinate $i$, we continue
to write $x=(n,v)$, where $n=x_i$ and $v=x_{-i}$.

Following the probability-current formulation of Ripsman, Kell, and Hilfinger \cite{RKH2025},
we define the stationary probability current across the $i$-birth edge
$(n,v)\to(n+1,v)$ by
\begin{equation}\label{eq:fiber-current}
J_i(n+1,n;v):=\pi(n,v)f_i(v)-\pi(n+1,v)\frac{n+1}{\tau_i}.
\end{equation}
Note that if the chain satisfies detailed balance, then each $J_i$ vanishes.  Although we do not assume detailed balance, stationarity does imply a weaker
marginal balance across each adjacent pair of levels of $X_i$. Namely,
\[
\sum_v J_i(n+1,n;v)=0.
\]
Equivalently,
\begin{equation}
\label{eq:marginal_flux_balance_preview}
\pi_i(n)m_i(n)=\frac{n+1}{\tau_i}\pi_i(n+1).
\end{equation}
Thus the net probability flux between the marginal levels $X_i=n$ and
$X_i=n+1$ vanishes, even though the individual conditional fluxes
$J_i(n+1,n;v)$ need not vanish. The  marginal balance \eqref{eq:marginal_flux_balance_preview} will be proved in Lemma \ref{lem:flux} below and will be used in the proof of part (3) of Theorem \ref{thm:main} below.

Also following \cite{RKH2025}, we define the single information flow as follows.
\begin{definition}[Single information flow]\label{def:Idot_clean}
Assume that the series below is absolutely convergent. Define
\begin{align}\label{eq:Idot_clean_def}
\begin{split}
(\dot I)_{X_i}
&:=\sum_{n\ge0}\sum_{v\in\Z_{\ge0}^{N-1}}
J_i(n+1,n;v)\,\log\frac{\pi_{n+1}^i(v)}{\pi_n^i(v)}\\
&=
\sum_{n\ge0}\sum_v
\left(
\pi(n,v)f_i(v)-\pi(n+1,v)\frac{n+1}{\tau_i}
\right)
\big(h_i(n,v)-h_i(n+1,v)\big),
\end{split}
\end{align}
where the second equality utilizes the notation from \eqref{eq:log_potentials}.
\end{definition}

Lemma \ref{lem:Idot_well_def_and_rep} below shows that under Assumption~\ref{ass:standing} the series is indeed absolutely convergent,
and gives a useful representation.

We now present the main theorem for the unit-birth model
\eqref{eq:generator}.

\begin{theorem}[Noise tradeoff for the unit-birth model]\label{thm:main}
Consider the unit-birth CTMC on $S=\Z_{\ge0}^N$ with generator
\eqref{eq:generator}. Suppose Assumption~\ref{ass:standing} holds, and let
$\pi$ be its stationary distribution. 
Then the following statements hold.
\begin{enumerate}[label=\textup{(\arabic*)}, leftmargin=2.2em]
\item For each $i$,
\begin{equation}\label{eq:mean_identity_thm}
\E_\pi[X_i]=\tau_i\,\mu_i,
\end{equation}
and
\begin{equation}\label{eq:fano_cov_identity_thm}
F_{X_i}
=1+\frac{\Cov_\pi\!\big(f_i(X_{-i}),X_i\big)}{\mu_i}.
\end{equation}

\item The single information flow $(\dot I)_{X_i}$ defined in
\eqref{eq:Idot_clean_def} is well defined, and for each $i$,
\begin{equation}\label{eq:single_information_ineq_thm}
\frac{F_{X_i}-1}{\tau_i}\ge (\dot I)_{X_i}.
\end{equation}

\item We have zero total information flow:
\begin{equation}\label{eq:zero_total_information_flow_thm}
\sum_{i=1}^N (\dot I)_{X_i}=0.
\end{equation}
\end{enumerate}
Consequently, we have the global tradeoff
\begin{equation}\label{eq:tradeoff_thm}
\sum_{i=1}^N \frac{F_{X_i}-1}{\tau_i}
\ge
\sum_{i=1}^N (\dot I)_{X_i}
=0.
\end{equation}
In particular, it is impossible to have $F_{X_i}\le 1$ for all $i$ with strict
inequality for at least one index: if $F_{X_k}<1$ for some $k$, then there
exists $j\neq k$ such that $F_{X_j}>1$.
\end{theorem}

\subsection{Sufficient conditions for Assumption \ref{ass:standing}}

We record a concrete
sufficient condition for Assumption~\ref{ass:standing}. The class is broad:
the individual birth-rate functions may be state dependent and unbounded, but
their total growth must be at most linear and not strong enough to overwhelm
the weakest linear death rate. The uniform lower bound ensures that all upward
moves remain accessible.

\begin{theorem}[Lower bounds and controlled growth imply Assumption~\ref{ass:standing}]
\label{thm:growth_f_implies_assumption}
Consider the continuous time Markov chain on $S=\Z_{\ge0}^N$ with generator
\eqref{eq:generator}. Assume there exists $\varepsilon\in(0,1)$ such that
\begin{equation}\label{eq:f_lower_bound}
\varepsilon \ \le\ f_i(v)\qquad\text{for all }i\in\{1,\dots,N\}\text{ and all }v\in\Z_{\ge0}^{N-1}.
\end{equation}
Assume further that there exist constants $A,B\in[0,\infty)$ such that, for all $x\in S$,
\begin{equation}\label{eq:total_birth_linear_growth}
\sum_{i=1}^N f_i(x_{-i}) \ \le\ A + B\,s(x),
\qquad s(x) = x_1 + \cdots + x_N,
\end{equation}
and that
\begin{equation}\label{eq:drift_gap_condition}
B \ <\ \frac{1}{\tau_{\max}},
\qquad \tau_{\max}:=\max_{1\le i\le N}\tau_i.
\end{equation}
Then Condition~\ref{condition:regularity_minimal} holds, and the corresponding
CTMC satisfies Assumption~\ref{ass:standing}.
\end{theorem}

\begin{proof}

First note that Condition~\ref{condition:regularity_minimal} holds. The rates
$q(x,y)$ are finite because the birth rates are nonnegative and satisfy
\[
\sum_{i=1}^N f_i(x_{-i})\le A+B s(x)<\infty,
\]
and the death rates $x_i/\tau_i$ are finite. The total rate out of $x$ is
finite for the same reason. Finally, for fixed $x$, there are at most $2N$
states $y$ with $q(y,x)>0$, namely $y=x-e_i$ or $y=x+e_i$ when these states
belong to $S$, and each corresponding rate is finite. Hence
\[
\sup_{y\neq x}q(y,x)<\infty.
\]
Thus Condition~\ref{condition:regularity_minimal} is satisfied.

We now verify items \Azero-\Athree of
Assumption~\ref{ass:standing}.

\smallskip\noindent
\emph{\Azero Irreducibility and stationarity, and
\Aone finite second moments.}

By \eqref{eq:f_lower_bound}, from any state $x$ and for every coordinate $i$,
the birth transition $x\to x+e_i$ occurs at strictly positive rate. Also,
whenever $x_i>0$, the death transition $x\to x-e_i$ occurs at rate
$x_i/\tau_i>0$. Hence any two states $x,y\in S$ can be connected by a finite
sequence of $\pm e_i$ moves, each having strictly positive rate. Therefore the
chain is irreducible on $S$.

Let
\[
V(x):=(1+s(x))^2.
\]
A birth increases $s$ by $1$, so it changes $V$ by
\[
(1+s(x)+1)^2-(1+s(x))^2=2s(x)+3.
\]
A death decreases $s$ by $1$, so it changes $V$ by
\[
(1+s(x)-1)^2-(1+s(x))^2=-2s(x)-1.
\]
Therefore
\begin{align*}
(LV)(x)
&=
(2s(x)+3)\sum_{i=1}^N f_i(x_{-i})
-
(2s(x)+1)\sum_{i=1}^N\frac{x_i}{\tau_i}.
\end{align*}
Using \eqref{eq:total_birth_linear_growth} and
\[
\sum_{i=1}^N\frac{x_i}{\tau_i}\ge \frac{s(x)}{\tau_{\max}},
\]
we get
\begin{align*}
(LV)(x)
&\le
(2s(x)+3)(A+B s(x))
-
(2s(x)+1)\frac{s(x)}{\tau_{\max}}\\
&=
\Bigl(2B-\frac{2}{\tau_{\max}}\Bigr)s(x)^2
+
\Bigl(2A+3B-\frac{1}{\tau_{\max}}\Bigr)s(x)
+
3A.
\end{align*}
By \eqref{eq:drift_gap_condition}, the coefficient of $s(x)^2$ is strictly
negative. Since $V(x)=(1+s(x))^2$, there exist constants $c>0$, $b<\infty$,
and $R<\infty$ such that
\[
(LV)(x)\le -c\,V(x)+b\,\mathbf 1_{\{s\le R\}}(x)
\qquad\forall x\in S.
\]
The set $\{x\in S:s(x)\le R\}$ is finite. Hence the standard
Foster--Lyapunov criterion for countable-state CTMCs (see \cite{DownMeynTweedie1995}) implies nonexplosion and
positive recurrence, and therefore existence of a stationary distribution
$\pi$; moreover,
\[
\E_\pi[V(X)]<\infty.
\]
Together with the irreducibility proved above, this proves \Azero.
Since $V(X)=(1+s(X))^2$, it also gives
\[
\E_\pi[s(X)^2]<\infty,
\]
which is equivalent to \Aone.

\smallskip\noindent
\emph{\Atwo Finite, positive mean birth rates.}
The lower bound \eqref{eq:f_lower_bound} gives
\[
\mu_i=\E_\pi[f_i(X_{-i})]\ge \varepsilon>0.
\]
For finiteness, \eqref{eq:total_birth_linear_growth} implies
\[
f_i(x_{-i})
\le
\sum_{j=1}^N f_j(x_{-j})
\le
A+B\,s(x).
\]
Therefore
\[
\mu_i
\le
A+B\,\E_\pi[s(X)]<\infty,
\]
where $\E_\pi[s(X)]<\infty$ follows from \Aone. Thus
\[
\mu_i\in(0,\infty)
\]
for each $i$, proving \Atwo.

\smallskip\noindent
\emph{\Athree At-most-linear total rate growth.}
For $x\in S$, the total transition rate is
\[
\Lambda(x)
=
\sum_{i=1}^N f_i(x_{-i})
+
\sum_{i=1}^N \frac{x_i}{\tau_i}.
\]
Using \eqref{eq:total_birth_linear_growth} and
\[
\sum_{i=1}^N\frac{x_i}{\tau_i}\le \frac{s(x)}{\tau_{\min}},
\qquad
\tau_{\min}:=\min_{1\le i\le N}\tau_i,
\]
we obtain
\[
\Lambda(x)
\le
A+B\,s(x)+\frac{s(x)}{\tau_{\min}}
\le
C_\Lambda(1+s(x)),
\]
for example with $C_\Lambda:=A+B+\tau_{\min}^{-1}$. This proves
\Athree.

Thus \Azero--\Athree hold, and so Assumption~\ref{ass:standing}
holds.
\end{proof}

%--------------------------------------------------------------
\section{Stationary  Identities and Logarithmic Integrability}
\label{sec:technical_lemmas}
%--------------------------------------------------------------

The proof of the main theorem relies on stationary generator identities for
unbounded logarithmic test functions. This section collects the estimates
needed to justify those identities on the infinite state space. We first record
general stationarity tools for countable-state continuous-time Markov chains,
and then apply them to the logarithmic potentials introduced above.

We first record that Assumption~\ref{ass:standing} gives the integrability of
the total exit rate needed below.

\begin{lemma}[$\Lambda$--integrability and nonexplosion]
\label{lem:nonexplosive_from_moments}
Consider the unit-birth CTMC on $S=\Z_{\ge0}^N$ with transition rates
$q(x,y)$ and total exit rate
\[
\Lambda(x):=\sum_{y\neq x}q(x,y).
\]
Suppose Assumption~\ref{ass:standing} holds. Then
\begin{equation}\label{eq:Lambda_pi_integrable}
\sum_{x\in S}\pi(x)\,\Lambda(x)=\E_\pi[\Lambda(X)]<\infty.
\end{equation}
In particular, by \cite[Theorem~2]{anderson2018non}, the CTMC is nonexplosive.
\end{lemma}

\begin{proof}
By \Athree,
\[
\E_\pi[\Lambda(X)] \le C_\Lambda\bigl(1+\E_\pi[s(X)]\bigr).
\]
By Jensen's inequality and \Aone,
\[
\E_\pi[s(X)] \le \bigl(\E_\pi[s(X)^2]\bigr)^{1/2}<\infty.
\]
Therefore $\E_\pi[\Lambda(X)]<\infty$, proving
\eqref{eq:Lambda_pi_integrable}. The nonexplosion conclusion follows from
\cite[Theorem~2]{anderson2018non}.
\end{proof}

The arguments below repeatedly use the stationarity identity
\[
\E_\pi[(L\varphi)(X)]=0
\]
for test functions that are not necessarily bounded.  The next lemma gives
natural hypotheses under which this identity is valid.   Although the ingredients
in the proof are standard, we have not found a reference that states the
particular form needed here. We therefore
include the short proof, both for completeness and to make clear exactly what
must be verified for the logarithmic test functions appearing in the
information-flow identities. We note that the lemma is a general statement for
continuous-time Markov chains on countable state spaces and is not specific to
the unit-birth model. For the unit-birth model, the stationarity, full-support, and
$\Lambda$--integrability hypotheses are supplied by
Assumption~\ref{ass:standing} and Lemma~\ref{lem:nonexplosive_from_moments};
the remaining task in later applications is to verify the displayed
$L^1(\pi)$ jump condition \eqref{eq:L1_assumption} for the particular test function under consideration.

\begin{lemma}[Stationarity identity for $L^1(\pi)$ test functions]
\label{lem:stationarity_L1}
Consider a CTMC on a countable state space $S$ with transition rates $q(x,y)$
and total exit rate $\Lambda(x)=\sum_{y\neq x}q(x,y)$. Suppose $\pi$ is a
stationary distribution with full support and
\[
\E_\pi[\Lambda(X)]<\infty.
\]
Let $\varphi:S\to\R$ satisfy
\begin{equation}\label{eq:L1_assumption}
\sum_{x\in S}\pi(x)\sum_{y\neq x}q(x,y)|\varphi(y)-\varphi(x)|<\infty.
\end{equation}
Then, for every $x\in S$, the series
\[
\sum_{y\neq x} q(x,y)\,|\varphi(y)-\varphi(x)|
\]
is finite, and hence
\[
(L\varphi)(x):=\sum_{y\neq x} q(x,y)\big(\varphi(y)-\varphi(x)\big)
\]
is absolutely convergent and well defined. Moreover, $L\varphi\in L^1(\pi)$
and
\[
\E_\pi\!\big[(L\varphi)(X)\big]=0.
\]
\end{lemma}

\begin{proof}
Set
\[
h(x):=\sum_{y\neq x}q(x,y)|\varphi(y)-\varphi(x)|.
\]
By the assumed $L^1(\pi)$ jump condition \eqref{eq:L1_assumption}, $h\in L^1(\pi)$. Since $\pi$ has full support, it follows that $h(x)<\infty$ for every $x\in S$. Hence, for every
$x\in S$, the series defining
\[
(L\varphi)(x):=\sum_{y\neq x}q(x,y)(\varphi(y)-\varphi(x))
\]
is absolutely convergent and well defined.  

We now prove the desired stationarity identity.  The argument has two steps:
first we prove the result for bounded truncations of $\varphi$, and then we
pass to the limit in $L^1(\pi)$.

For $M\in\N$, let
\[
\varphi^{(M)}:=(-M)\vee(\varphi\wedge M).
\]
Then $|\varphi^{(M)}|\le M$, and hence
\[
|(L\varphi^{(M)})(x)|
\le
\sum_{y\neq x}q(x,y)|\varphi^{(M)}(y)-\varphi^{(M)}(x)|
\le 2M\Lambda(x).
\]
Since $\E_\pi[\Lambda(X)]<\infty$, this shows that
$L\varphi^{(M)}\in L^1(\pi)$.

We claim that
\[
\E_\pi[(L\varphi^{(M)})(X)]=0.
\]
Indeed, since $|\varphi^{(M)}|\le M$ and $\E_\pi[\Lambda(X)]<\infty$, the sums
below are absolutely convergent, and therefore
\begin{align*}
\E_\pi[(L\varphi^{(M)})(X)]
&=
\sum_{x\in S}\pi(x)\sum_{y\neq x}q(x,y)
\big(\varphi^{(M)}(y)-\varphi^{(M)}(x)\big)\\
&=
\sum_{x\in S}\sum_{y\neq x}\pi(x)q(x,y)\varphi^{(M)}(y)
-
\sum_{x\in S}\pi(x)\Lambda(x)\varphi^{(M)}(x).
\end{align*}
Interchanging the order of summation in the first term gives
\[
\sum_{x\in S}\sum_{y\neq x}\pi(x)q(x,y)\varphi^{(M)}(y)
=
\sum_{y\in S}\varphi^{(M)}(y)\sum_{x\neq y}\pi(x)q(x,y).
\]
Since $\pi$ is stationary and $\E_\pi[\Lambda(X)]<\infty$, the stationary
forward equation holds entry-by-entry:
\[
\sum_{x\neq y}\pi(x)q(x,y)=\pi(y)\Lambda(y),
\qquad y\in S,
\]
see, for example, Norris~\cite[Section~3.5, p.~117]{norris1998markov}. Thus
the two terms above are equal, and so
\begin{equation}\label{eq:step1}
\E_\pi[(L\varphi^{(M)})(X)]=0
\qquad\text{for every }M\in\N.
\end{equation}

It remains to pass to the limit. For $x,y\in S$ with $y\neq x$, define
\[
D_M(x,y)
:=
\big(\varphi^{(M)}(y)-\varphi^{(M)}(x)\big)
-
\big(\varphi(y)-\varphi(x)\big).
\]
For each fixed $x$ and $y\neq x$, we have $D_M(x,y)\to0$ as $M\to\infty$.
Moreover, since the truncation map $r\mapsto (-M)\vee(r\wedge M)$ is
$1$-Lipschitz, we have $|\varphi^{(M)}(y)-\varphi^{(M)}(x)| \le |\varphi(y)-\varphi(x)|$ and so
\[
|D_M(x,y)|
\le
2|\varphi(y)-\varphi(x)|, \quad \text{ for all  $x$ and $y \ne x$.}
\]
Hence
\[
q(x,y)|D_M(x,y)|
\le
2q(x,y)|\varphi(y)-\varphi(x)|.
\]
For fixed $x$, the right-hand side is summable in $y$, since
\[
\sum_{y\neq x}2q(x,y)|\varphi(y)-\varphi(x)|
=
2h(x)<\infty.
\]
Therefore dominated convergence, applied to the counting measure sum over
$y\neq x$, gives
\[
(L\varphi^{(M)})(x)-(L\varphi)(x)
=
\sum_{y\neq x}q(x,y)D_M(x,y)
\to 0
\qquad\text{for every }x\in S, \text{ as $M\to \infty$}.
\]
The same estimate also gives, for every $x\in S$,
\[
|(L\varphi^{(M)})(x)-(L\varphi)(x)|
\le
\sum_{y\neq x}q(x,y)|D_M(x,y)|
\le
2h(x).
\]
Since $h\in L^1(\pi)$, dominated convergence with respect to $\pi$ yields
\begin{equation}\label{eq:786978}
L\varphi^{(M)}-L\varphi\to0
\qquad\text{in }L^1(\pi).
\end{equation}
In particular, $L\varphi\in L^1(\pi)$. Combining \eqref{eq:step1} and
\eqref{eq:786978}, we obtain
\[
\big|\E_\pi[(L\varphi)(X)]\big|
\le
\E_\pi\!\left[
\left|(L\varphi)(X)-(L\varphi^{(M)})(X)\right|
\right]
\longrightarrow 0.
\]
Thus $\E_\pi[(L\varphi)(X)]=0$, as claimed.
\end{proof}

\subsection{Logarithmic integrability estimates}

The next estimates are stated for the lattice state space
$S=\Z_{\ge0}^N$, but they do not use the unit-birth structure except through
the moment and total-rate growth assumptions. In particular, the weighted
entropy estimate is a statement about the stationary law $\pi$, while the
subsequent generator estimate applies to any CTMC on $S$ whose total exit rate
$\Lambda$ satisfies the corresponding integrability bounds. In the unit-birth
model, these hypotheses are supplied by Assumption~\ref{ass:standing}.

\begin{lemma}[Weighted entropy tail and $\Lambda$--weighted integrability]
\label{lem:weighted_entropy_tail}
Let $\pi$ be a probability distribution on $S=\Z_{\ge0}^N$ with full support satisfying
$\E_\pi[s(X)^2]<\infty$, where $s(x)=\|x\|_1$. Then
\begin{equation}\label{eq:weighted_entropy_tail}
\sum_{x\in S}\pi(x)\,(1+s(x))\,|\log\pi(x)|<\infty.
\end{equation}
Moreover, if $\Lambda:S\to[0,\infty)$ satisfies
\begin{equation}
\label{eq:Lambda_linear_growth_lem}
\Lambda(x)\le C_\Lambda(1+s(x))\qquad \forall x\in S,
\end{equation}
then
\begin{equation}\label{eq:Lambda_log_integrable}
\sum_{x\in S}\pi(x)\,\Lambda(x)\,\bigl(1+|\log\pi(x)|\bigr)<\infty.
\end{equation}
\end{lemma}

\begin{proof}
Note that under \eqref{eq:Lambda_linear_growth_lem}, the inequality \eqref{eq:Lambda_log_integrable}
follows immediately from \eqref{eq:weighted_entropy_tail}. Hence it suffices to prove
\eqref{eq:weighted_entropy_tail}.

Fix $\rho\in(0,1)$ and define a reference probability measure on $S$ by
\[
Q_\rho(x):=(1-\rho)^N\rho^{s(x)},\qquad x\in S.
\]
Let $\beta=\E_\pi[1+s(X)]$, which is finite by $\E_\pi[s(X)^2]<\infty$ and Jensen's inequality, and define the ``tilted'' probability measure
\[
\pi^\ast(x):=\frac{(1+s(x))\pi(x)}{\beta}.
\]
Since $Q_\rho(x)>0$ for all $x$, we have $\pi^\ast\ll Q_\rho$, and hence
$D_{\mathrm{KL}}(\pi^\ast\|Q_\rho)\ge 0$, i.e.
\[
-\sum_x \pi^\ast(x)\log \pi^\ast(x)\ \le\ -\sum_x \pi^\ast(x)\log Q_\rho(x).
\]
Now $\log Q_\rho(x)=N\log(1-\rho)+s(x)\log\rho$, so
\[
-\sum_x \pi^\ast(x)\log Q_\rho(x)
= -N\log(1-\rho) -(\log\rho)\,\E_{\pi^\ast}[s(X)].
\]
Moreover,
\[
\E_{\pi^\ast}[s(X)]
=\frac{\E_\pi[(1+s(X))s(X)]}{\beta}\le \frac{\E_\pi[(1+s(X))^2]}{\beta}
\le \frac{\E_\pi[2(1+s(X)^2)]}{\beta}
<\infty,
\]
where finiteness follows from the assumed finite second moment.  Therefore
\[
H(\pi^\ast):=-\sum_x \pi^\ast(x)\log\pi^\ast(x)<\infty.
\]

Next,  note that
\[
\log\pi^\ast(x)=\log\pi(x)+\log(1+s(x))-\log \beta
\]
which we rearrange and use the  triangle inequality to obtain
\[
|\log\pi(x)|
\le |\log\pi^\ast(x)|+\log(1+s(x))+|\log \beta|.
\]
Multiplying by $(1+s(x))\pi(x)$ and summing over $x$ yields
\begin{align*}
\sum_x \pi(x)(1+s(x))&\,|\log\pi(x)|
\le \sum_x \pi(x)(1+s(x))\,|\log\pi^\ast(x)|
   + \sum_x \pi(x)(1+s(x))\,\log(1+s(x))\\
   &\hspace{1in}
   + |\log \beta|\sum_x \pi(x)(1+s(x))\\
&= \beta \sum_x \pi^\ast(x)\,|\log\pi^\ast(x)|
   + \sum_x \pi(x)(1+s(x))\,\log(1+s(x))
   + |\log \beta|\,\beta.
\end{align*}

Since $H(\pi^\ast)<\infty$ implies $\sum_x \pi^\ast(x)|\log\pi^\ast(x)|<\infty$, the first term is finite.
For the second term, use the elementary bound $(1+r)\log(1+r)\le (1+r^2)$ for $r\ge 0$ and the assumed finite second moment to conclude finiteness. This proves \eqref{eq:weighted_entropy_tail}.
\end{proof}

We next apply the preceding estimate to the logarithmic potential
$g(x)=-\log\pi(x)$. This is the first place where the general stationarity
Lemma~\ref{lem:stationarity_L1} above is used: the goal is to verify its
$L^1(\pi)$ jump condition \eqref{eq:L1_assumption} for $g$, and hence justify
the identity $\E_\pi[(Lg)(X)]=0$.

\begin{lemma}[Integrability for $g=-\log\pi$ and stationarity of $Lg$]
\label{lem:abs_int_Lg}
Consider a CTMC on $S=\Z_{\ge0}^N$ for which
Assumption~\ref{ass:standing} holds. Let
$\Lambda(x):=\sum_{y\neq x}q(x,y)$ and define $g:S\to\R$ by
\[
g(x):=-\log\pi(x).
\]
Then
\begin{equation}\label{eq:g_weighted_double_sum_finite}
\sum_{x\in S}\pi(x)\sum_{y\neq x}q(x,y)\bigl(|g(y)|+|g(x)|\bigr)<\infty.
\end{equation}
In particular, $Lg\in L^1(\pi)$ and
\[
\E_\pi[(Lg)(X)]=0.
\]
\end{lemma}

\begin{proof}
By Lemma~\ref{lem:nonexplosive_from_moments}, we have
$\sum_{x\in S}\pi(x)\Lambda(x)<\infty$. Moreover, by
Lemma~\ref{lem:weighted_entropy_tail} and the linear growth assumption
\Athree,
\begin{equation}\label{eq:87976978}
\sum_{x\in S}\pi(x)\Lambda(x)|g(x)|
=
\sum_{x\in S}\pi(x)\Lambda(x)|\log\pi(x)|<\infty.
\end{equation}

Write
\[
T:=\sum_{x\in S}\pi(x)\sum_{y\neq x}q(x,y)\bigl(|g(y)|+|g(x)|\bigr)=T_1+T_2,
\]
where
\[
T_1:=\sum_{x\in S}\pi(x)\sum_{y\neq x}q(x,y)\,|g(x)|
=\sum_{x\in S}\pi(x)\Lambda(x)\,|g(x)|
\]
and
\[
T_2:=\sum_{x\in S}\pi(x)\sum_{y\neq x}q(x,y)\,|g(y)|.
\]
By \eqref{eq:87976978}, $T_1<\infty$. Since all terms are nonnegative, Tonelli's theorem permits swapping sums in $T_2$:
\[
T_2=\sum_{y\in S}|g(y)|\sum_{x\in S,\,x\neq y}\pi(x)q(x,y).
\]
Because $\pi$ is stationary, for each $y\in S$,
\[
\sum_{x\in S,\,x\neq y}\pi(x)q(x,y)=\pi(y)\sum_{z\in S,\,z\neq y}q(y,z)=\pi(y)\Lambda(y),
\]
and hence
\[
T_2=\sum_{y\in S}\pi(y)\Lambda(y)|g(y)|=T_1<\infty.
\]
Therefore $T=T_1+T_2<\infty$, proving \eqref{eq:g_weighted_double_sum_finite}.

Next, since $|g(y)-g(x)|\le |g(y)|+|g(x)|$, we have
\[
\sum_{x\in S}\pi(x)\sum_{y\neq x}q(x,y)\,|g(y)-g(x)|
\le T<\infty.
\]
Thus the hypotheses of Lemma~\ref{lem:stationarity_L1} (with $\varphi=g$) hold, and we conclude
$Lg\in L^1(\pi)$ and $\E_\pi[(Lg)(X)]=0$.
\end{proof}

\subsection{Logarithmic potentials, single information flow, and marginal flux balance}

Section~\ref{sec:statement_main_theorem} introduced the fiber currents
$J_i(n+1,n;v)$ and the single information flow $(\dot I)_{X_i}$. We now prove
the two unit-birth identities promised there. First, we show that the series
defining $(\dot I)_{X_i}$ is absolutely convergent and equals the
coordinate-generator quantity $-\E_\pi[(L_i h_i)(X)]$. Second, we prove the
marginal flux balance \eqref{eq:marginal_flux_balance_preview}, equivalently
$\sum_v J_i(n+1,n;v)=0$. These identities form the technical bridge between
the information-flow expression and the generator formalism used in the proof
of Theorem~\ref{thm:main}.

\begin{lemma}[Single information flow: absolute convergence and generator representation]
\label{lem:Idot_well_def_and_rep}
Fix $i \in \{1,\dots, N\}$.
For the unit-birth model with generator \eqref{eq:generator}, suppose
Assumption~\ref{ass:standing} holds. Then the series defining
$(\dot I)_{X_i}$ in \eqref{eq:Idot_clean_def} is absolutely convergent.

More precisely, writing
\[
(\dot I)_{X_i}
=
\sum_{n\ge0}\sum_v
\left(
\pi(n,v)f_i(v)-\pi(n+1,v)\frac{n+1}{\tau_i}
\right)
\big(h_i(n,v)-h_i(n+1,v)\big),
\]
each of the four nonnegative sums
\begin{align}
\label{eq:four_piece_bounds}
\sum_{n,v}\pi(n,v)f_i(v)\,h_i(n+1,v),
\qquad
\sum_{n,v}\pi(n,v)f_i(v)\,h_i(n,v),\\
\notag
\sum_{n,v}\pi(n+1,v)\frac{n+1}{\tau_i}\,h_i(n+1,v),
\qquad
\sum_{n,v}\pi(n+1,v)\frac{n+1}{\tau_i}\,h_i(n,v)
\end{align}
is finite. Consequently, $L_i h_i\in L^1(\pi)$ and
\begin{equation}\label{eq:Idot_equals_Lih}
(\dot I)_{X_i}
=
-\E_\pi\big[(L_i h_i)(X)\big].
\end{equation}
\end{lemma}

\begin{proof}
We first record the pointwise comparison between $h_i$ and $g$. Since
\[
\pi_n^i(v)=\frac{\pi(n,v)}{\pi_i(n)}
\qquad\text{and}\qquad
\pi_i(n)\le 1,
\]
we have $\pi_n^i(v)\ge \pi(n,v)$. Therefore
\begin{equation}\label{eq:hi_le_g_pointwise_Idot}
0\le h_i(n,v)\le g(n,v)
\qquad\text{for all }(n,v)\in S.
\end{equation}

By Lemma~\ref{lem:abs_int_Lg},
\begin{equation}\label{eq:abs_int_Lg_double_sum_for_Idot}
\sum_{x\in S}\pi(x)\sum_{y\neq x}q(x,y)\big(|g(y)|+|g(x)|\big)<\infty.
\end{equation}
We now write the relevant pieces of this bound in the present $(n,v)$ notation.
The $i$--birth transitions are
\[
(n,v)\longrightarrow(n+1,v)
\qquad\text{with rate } f_i(v),
\]
and their contribution to \eqref{eq:abs_int_Lg_double_sum_for_Idot} is
\[
\sum_{n,v}\pi(n,v)f_i(v)
\bigl(|g(n+1,v)|+|g(n,v)|\bigr)<\infty.
\]
Similarly, the $i$--death transitions may be indexed as
\[
(n+1,v)\longrightarrow(n,v)
\qquad\text{with rate } \frac{n+1}{\tau_i},
\]
and their contribution to \eqref{eq:abs_int_Lg_double_sum_for_Idot} is
\[
\sum_{n,v}\pi(n+1,v)\frac{n+1}{\tau_i}
\bigl(|g(n,v)|+|g(n+1,v)|\bigr)<\infty.
\]
Since $g\ge0$ and $0\le h_i\le g$ by
\eqref{eq:hi_le_g_pointwise_Idot}, these two finite bounds show that the four
sums in \eqref{eq:four_piece_bounds} are finite.
It follows that the series defining $(\dot I)_{X_i}$ is absolutely convergent,
since
\[
\sum_{n\ge0}\sum_v\left|
\pi(n,v)f_i(v)-\pi(n+1,v)\frac{n+1}{\tau_i}
\right|
\left|h_i(n,v)-h_i(n+1,v)\right|
\]
is bounded above by the sum of the four nonnegative sums appearing in
\eqref{eq:four_piece_bounds}.

The same four bounds also imply
\[
\sum_{n,v}\pi(n,v)f_i(v)|h_i(n+1,v)-h_i(n,v)|
+
\sum_{n\ge1,v}\pi(n,v)\frac{n}{\tau_i}|h_i(n-1,v)-h_i(n,v)|
<\infty.
\]
Thus $L_i h_i\in L^1(\pi)$, and the following termwise computation is justified:
\begin{align*}
\E_\pi[(L_i h_i)(X)]
&=
\sum_{n,v}\pi(n,v)f_i(v)\big(h_i(n+1,v)-h_i(n,v)\big)\\
&\quad+
\sum_{n\ge1,v}\pi(n,v)\frac{n}{\tau_i}
\big(h_i(n-1,v)-h_i(n,v)\big).
\end{align*}
Reindexing the death term by replacing $n$ with $n+1$ gives
\[
\sum_{n\ge1,v}\pi(n,v)\frac{n}{\tau_i}
\big(h_i(n-1,v)-h_i(n,v)\big)
=
\sum_{n\ge0,v}\pi(n+1,v)\frac{n+1}{\tau_i}
\big(h_i(n,v)-h_i(n+1,v)\big).
\]
Combining the paired birth and death terms gives
\begin{align*}
\E_\pi[(L_i h_i)(X)]
&=
-\sum_{n\ge0}\sum_v
\left(
\pi(n,v)f_i(v)-\pi(n+1,v)\frac{n+1}{\tau_i}
\right)
\big(h_i(n,v)-h_i(n+1,v)\big).
\end{align*}
By Definition~\ref{def:Idot_clean}, the sum on the right-hand side is
$(\dot I)_{X_i}$. Hence $\E_\pi[(L_i h_i)(X)]=-(\dot I)_{X_i}$
as claimed.
\end{proof}

We also prove the marginal flux balance stated in Section~\ref{sec:statement_main_theorem}. This is weaker than detailed balance: the individual fiber currents
$J_i(n+1,n;v)$ need not vanish, but their sum over each fiber does.

\begin{lemma}[Marginal flux balance]\label{lem:flux}
Consider the unit-birth CTMC on $S=\Z_{\ge0}^N$ with generator
\eqref{eq:generator}. Suppose Assumption~\ref{ass:standing} holds, and let
$\pi$ be its stationary distribution. Then, for every $i\in\{1,\dots,N\}$ and
$n\ge0$,
\begin{equation}\label{eq:marginal-flux}
\pi_i(n)\,m_i(n)=\frac{n+1}{\tau_i}\,\pi_i(n+1).
\end{equation}
Equivalently,
\[
\sum_v J_i(n+1,n;v)=0.
\]
\end{lemma}

\begin{proof}
Fix $i$ and $n\ge0$, and let
\[
\varphi(x):=\mathbf 1_{\{x_i\le n\}}.
\]
Since $\varphi$ is bounded, the $L^1(\pi)$ jump condition in
Lemma~\ref{lem:stationarity_L1} is bounded by $2\E_\pi[\Lambda(X)]<\infty$.
Thus Lemma~\ref{lem:stationarity_L1} applies, and hence
\[
\E_\pi[(L\varphi)(X)]=0.
\]
Only $i$--births from level $n$ and $i$--deaths from level $n+1$ change
$\varphi$, so
\[
(L\varphi)(x)
=
-f_i(x_{-i})\mathbf 1_{\{x_i=n\}}
+
\frac{n+1}{\tau_i}\mathbf 1_{\{x_i=n+1\}}.
\]
Taking expectation under $\pi$ gives
\[
0
=
-\sum_v\pi(n,v)f_i(v)
+
\frac{n+1}{\tau_i}\sum_v\pi(n+1,v),
\]
which is exactly \eqref{eq:marginal-flux}. The equivalent form
$\sum_v J_i(n+1,n;v)=0$ follows directly from the definition of
$J_i(n+1,n;v)$.
\end{proof}

\section{Proof of Theorem~\ref{thm:main}}
\label{sec:proof}

We prove Theorem~\ref{thm:main} according to its numbered claims. The
technical estimates in Section~\ref{sec:technical_lemmas} justify the
stationary generator identities for the unbounded test functions used below.
Throughout this section we work under Assumption~\ref{ass:standing}.

\subsection{Part \textup{(1)}: stationary mean and Fano--covariance identity}

\begin{proof}[Proof of Theorem~\ref{thm:main} part \textup{(1)}]
Fix $i\in\{1,\dots,N\}$ and write $X\sim\pi$.

First let $\varphi(x)=x_i$. Since only jumps in the $i$th coordinate change
$\varphi$, we have
\[
(L\varphi)(x)
=
f_i(x_{-i})-\frac{x_i}{\tau_i}.
\]
Moreover,
\[
\sum_{y\neq x}q(x,y)|\varphi(y)-\varphi(x)|
=
f_i(x_{-i})+\frac{x_i}{\tau_i}.
\]
Taking $\pi$--expectations, this is finite by \Atwo and \Aone. Hence
Lemma~\ref{lem:stationarity_L1} applies and gives
\[
0=\E_\pi[(L\varphi)(X)]
=
\E_\pi[f_i(X_{-i})]-\frac{1}{\tau_i}\E_\pi[X_i]
=
\mu_i-\frac{1}{\tau_i}\E_\pi[X_i].
\]
Therefore
\begin{equation}\label{eq:mean_identity_thm_proof}
\E_\pi[X_i]=\tau_i\mu_i,
\end{equation}
which proves \eqref{eq:mean_identity_thm}.

Next let $\varphi(x)=x_i^2$. Again only jumps in the $i$th coordinate
contribute, and
\begin{align*}
(L\varphi)(x)
&=
f_i(x_{-i})\big((x_i+1)^2-x_i^2\big)
+
\frac{x_i}{\tau_i}\big((x_i-1)^2-x_i^2\big)\\
&=
f_i(x_{-i})(2x_i+1)+\frac{x_i}{\tau_i}(-2x_i+1).
\end{align*}
To justify stationarity, note that
\[
|\,(x_i+1)^2-x_i^2\,|=2x_i+1,
\qquad
|\,(x_i-1)^2-x_i^2\,|\le 2x_i+1,
\]
and therefore
\[
\sum_{y\neq x}q(x,y)|\varphi(y)-\varphi(x)|
\le
(2x_i+1)\Big(f_i(x_{-i})+\frac{x_i}{\tau_i}\Big)
\le
(2x_i+1)\Lambda(x).
\]
Using \Athree and $x_i\le s(x)$,
\[
(2x_i+1)\Lambda(x)
\le
2C_\Lambda(1+s(x))(1+x_i)
\le
2C_\Lambda(1+s(x))^2
\le
4C_\Lambda(1+s(x)^2).
\]
This is $\pi$--integrable by \Aone, so Lemma~\ref{lem:stationarity_L1} applies.
Thus
\[
0=\E_\pi[(L\varphi)(X)]
=
\E_\pi\!\big[f_i(X_{-i})(2X_i+1)\big]
+
\E_\pi\!\left[\frac{X_i}{\tau_i}(-2X_i+1)\right].
\]
Equivalently,
\[
0
=
2\E_\pi[f_i(X_{-i})X_i]+\mu_i
-\frac{2}{\tau_i}\E_\pi[X_i^2]
+\frac{1}{\tau_i}\E_\pi[X_i].
\]
Using \eqref{eq:mean_identity_thm_proof}, this becomes
\[
\E_\pi[X_i^2]
=
\tau_i\E_\pi[f_i(X_{-i})X_i]+\E_\pi[X_i].
\]
Subtracting $\E_\pi[X_i]^2
=
\tau_i\E_\pi[f_i(X_{-i})]\,\E_\pi[X_i]$
from both sides gives
\[
\Var_\pi(X_i)
=
\tau_i\Cov_\pi\!\big(f_i(X_{-i}),X_i\big)
+
\E_\pi[X_i].
\]
Finally, dividing by $\E_\pi[X_i]=\tau_i\mu_i>0$ yields
\[
F_{X_i}
=
1+\frac{\Cov_\pi\!\big(f_i(X_{-i}),X_i\big)}{\mu_i},
\]
which is \eqref{eq:fano_cov_identity_thm}.
\end{proof}

\subsection{Part \textup{(2)}: single-component information-flow inequality}

The argument in this subsection follows the formal log-sum estimate of
Ripsman, Kell, and Hilfinger~\cite{RKH2025}.  The new point here is that, by
the results of Section~\ref{sec:technical_lemmas}, all sums and reindexings
below are justified under Assumption~\ref{ass:standing}.

\begin{proof}[Proof of Theorem~\ref{thm:main} part \textup{(2)}]
Fix $i\in\{1,\dots,N\}$. By Lemma~\ref{lem:Idot_well_def_and_rep}, the
series defining $(\dot I)_{X_i}$ is absolutely convergent. Write
$(\dot I)_{X_i}=A+B$, where
\begin{align*}
A&:=\sum_{n\ge0}\sum_{v}\pi(n,v)f_i(v)
\log\frac{\pi^i_{n+1}(v)}{\pi^i_n(v)},\\
B&:=-\sum_{n\ge0}\sum_{v}\pi(n+1,v)\frac{n+1}{\tau_i}
\log\frac{\pi^i_{n+1}(v)}{\pi^i_n(v)}.
\end{align*}

We first show that the second term is nonpositive. Fix $n\ge0$ and set
\[
b_n:=\pi_i(n+1)\frac{n+1}{\tau_i}.
\]
Since $\pi$ has full support, $b_n>0$. Then
\[
\sum_v \pi(n+1,v)\frac{n+1}{\tau_i}
\log\frac{\pi^i_{n+1}(v)}{\pi^i_n(v)}
=
b_n\sum_v \pi^i_{n+1}(v)
\log\frac{\pi^i_{n+1}(v)}{\pi^i_n(v)}
=
b_n\,D_{\mathrm{KL}}(\pi^i_{n+1}\|\pi^i_n)\ge0.
\]
Hence $B\le0$, and therefore
\begin{equation}\label{eq:dropB}
(\dot I)_{X_i}\le A.
\end{equation}

We next bound $A$ by applying the log-sum inequality level by level. Fix
$n\ge0$ and define
\[
b_v:=\pi(n,v)f_i(v)=\pi_i(n)\pi^i_n(v)f_i(v),
\qquad
a_v:=\pi_i(n)\pi^i_{n+1}(v)f_i(v).
\]
If $f_i(v)=0$, then $a_v=b_v=0$, and the corresponding term is interpreted as
zero. If $f_i(v)>0$, then
\[
\frac{a_v}{b_v}=\frac{\pi^i_{n+1}(v)}{\pi^i_n(v)}.
\]
Thus, writing
\[
\widehat A_n
:=
\sum_v \pi(n,v)f_i(v)\log\frac{\pi^i_{n+1}(v)}{\pi^i_n(v)}
=
\sum_v b_v\log\frac{a_v}{b_v},
\]
so that $A = \sum_{n \ge 0} \widehat A_n$, 
the log-sum inequality gives
\[
\widehat A_n
\le
\left(\sum_v b_v\right)
\log\frac{\sum_v a_v}{\sum_v b_v}.
\]
Since
\[
\sum_v b_v=\pi_i(n)m_i(n),
\qquad
\sum_v a_v=\pi_i(n)m_i(n+1),
\]
with each term positive by \eqref{eq:marginal-flux},
we obtain
\[
\widehat A_n
\le
\pi_i(n)m_i(n)\log\frac{m_i(n+1)}{m_i(n)}.
\]
Summing over $n$ gives
\begin{equation}\label{eq:A_bd_by_mi}
A\le
\sum_{n\ge0}\pi_i(n)m_i(n)\log\frac{m_i(n+1)}{m_i(n)}.
\end{equation}

We now estimate the right-hand side of \eqref{eq:A_bd_by_mi}. Multiplying and
dividing by $\mu_i$, we write
\[
\log\frac{m_i(n+1)}{m_i(n)}
=
\log\frac{m_i(n+1)}{\mu_i}
-
\log\frac{m_i(n)}{\mu_i}.
\]
Using $\log r\le r-1$ and $-\log r\le r^{-1}-1$ for $r>0$ gives
\begin{align*}
&\pi_i(n)m_i(n)
\left(
\log\frac{m_i(n+1)}{\mu_i}
-
\log\frac{m_i(n)}{\mu_i}
\right)\\
&\qquad\le
\pi_i(n)\left(\frac{m_i(n)m_i(n+1)}{\mu_i}-m_i(n)\right)
+
\pi_i(n)\big(\mu_i-m_i(n)\big).
\end{align*}
Summing over $n$ and using
\[
\sum_{n\ge0}\pi_i(n)m_i(n)=\mu_i
\]
yields
\[
A\le
\frac{1}{\mu_i}\sum_{n\ge0}\pi_i(n)m_i(n)m_i(n+1)-\mu_i.
\]
By the marginal flux balance identity \eqref{eq:marginal-flux},
\[
\pi_i(n)m_i(n)=\frac{n+1}{\tau_i}\pi_i(n+1),
\]
and hence
\begin{align*}
\sum_{n\ge0}\pi_i(n)m_i(n)m_i(n+1)
&=
\sum_{n\ge0}\frac{n+1}{\tau_i}\pi_i(n+1)m_i(n+1)\\
&=
\frac{1}{\tau_i}\sum_{n\ge1}n\,\pi_i(n)m_i(n)\\
&=
\frac{1}{\tau_i}\E_\pi[X_i\,m_i(X_i)].
\end{align*}
Therefore
\[
A\le
\frac{1}{\tau_i\mu_i}\E_\pi[X_i\,m_i(X_i)]-\mu_i.
\]
Combining this with \eqref{eq:dropB} gives
\[
(\dot I)_{X_i}
\le
\frac{1}{\tau_i\mu_i}\E_\pi[X_i\,m_i(X_i)]-\mu_i.
\]
Using $\E_\pi[X_i]=\tau_i\mu_i$ from part \textup{(1)} and
$\E_\pi[m_i(X_i)]=\mu_i$, the right-hand side becomes
\[
\frac{\Cov_\pi(X_i,m_i(X_i))}{\tau_i\mu_i}.
\]
Finally, since
\[
m_i(X_i)=\E_\pi[f_i(X_{-i})\mid X_i],
\]
we have
\[
\Cov_\pi(X_i,m_i(X_i))
=
\Cov_\pi\!\big(X_i,f_i(X_{-i})\big).
\]
Thus
\[
(\dot I)_{X_i}
\le
\frac{\Cov_\pi\!\big(f_i(X_{-i}),X_i\big)}{\tau_i\mu_i}.
\]
By the Fano--covariance identity from part \textup{(1)},
\[
F_{X_i}-1
=
\frac{\Cov_\pi\!\big(f_i(X_{-i}),X_i\big)}{\mu_i}.
\]
Therefore
\[
(\dot I)_{X_i}
\le
\frac{F_{X_i}-1}{\tau_i},
\]
which is exactly \eqref{eq:single_information_ineq_thm}.
\end{proof}

\subsection{Part \textup{(3)}: zero total information flow}

This cancellation is the rigorous generator version of the formal entropy
decomposition used by Ripsman, Kell, and Hilfinger in \cite{RKH2025}: the global logarithmic
potential is split into marginal and conditional pieces, the marginal
contribution vanishes by marginal flux balance, and the conditional
contributions are precisely the single information flows.

\begin{proof}[Proof of Theorem~\ref{thm:main} part \textup{(3)}]
Recall the logarithmic potentials from \eqref{eq:log_potentials}. For each
$i$,
\[
g(x)=g_i(x)+h_i(x),\qquad x\in S.
\]
By Lemma~\ref{lem:abs_int_Lg},
\[
Lg\in L^1(\pi)
\qquad\text{and}\qquad
\E_\pi[(Lg)(X)]=0.
\]

Recalling $L_i$ from \eqref{eq:Li_def}, we next show that
\[
\E_\pi[(L_i g_i)(X)]=0
\qquad\text{for each }i.
\]
Fix $i$. When writing $x=(n,v)$, we use the shorthand
$g_i(n):=-\log \pi_i(n)$. Since $\pi_i(n)\ge \pi(n,v)$, we have
\[
0\le g_i(n)\le g(n,v)
\qquad\text{for all }(n,v)\in S.
\]
Restricting the integrability bound from Lemma~\ref{lem:abs_int_Lg} to the
$i$--birth and $i$--death transitions therefore gives
\[
\sum_{n,v}\pi(n,v)f_i(v)|g_i(n+1)-g_i(n)|
+
\sum_{n\ge1,v}\pi(n,v)\frac{n}{\tau_i}|g_i(n-1)-g_i(n)|
<\infty.
\]
Thus $L_i g_i\in L^1(\pi)$, and the following termwise computation is justified:
\begin{align*}
\E_\pi[(L_i g_i)(X)]
&=
\sum_{n,v}\pi(n,v)f_i(v)\bigl(g_i(n+1)-g_i(n)\bigr)\\
&\quad+
\sum_{n\ge1,v}\pi(n,v)\frac{n}{\tau_i}
\bigl(g_i(n-1)-g_i(n)\bigr)\\
&=
\sum_{n\ge0}\pi_i(n)m_i(n)\bigl(g_i(n+1)-g_i(n)\bigr)\\
&\quad+
\sum_{n\ge1}\pi_i(n)\frac{n}{\tau_i}
\bigl(g_i(n-1)-g_i(n)\bigr).
\end{align*}
Reindexing the second sum by replacing $n$ with $n+1$ gives
\[
\sum_{n\ge1}\pi_i(n)\frac{n}{\tau_i}
\bigl(g_i(n-1)-g_i(n)\bigr)
=
\sum_{n\ge0}\pi_i(n+1)\frac{n+1}{\tau_i}
\bigl(g_i(n)-g_i(n+1)\bigr).
\]
Combining the two sums,
\[
\E_\pi[(L_i g_i)(X)]
=
\sum_{n\ge0}
\left(
\pi_i(n)m_i(n)-\frac{n+1}{\tau_i}\pi_i(n+1)
\right)
\bigl(g_i(n+1)-g_i(n)\bigr).
\]
By the marginal flux balance Lemma~\ref{lem:flux}, the coefficient in
parentheses is zero for every $n$. Hence
\[
\E_\pi[(L_i g_i)(X)]=0.
\]

Finally, since $L=\sum_{i=1}^N L_i$ and $g=g_i+h_i$ for each $i$ (from \eqref{eq:g_decomp}), we have the
pointwise identity
\[
Lg
= \sum_{i = 1}^N (L_i g) = 
\sum_{i=1}^N L_i g_i+\sum_{i=1}^N L_i h_i.
\]
All terms are integrable by the preceding paragraph,
Lemma~\ref{lem:abs_int_Lg}, and Lemma~\ref{lem:Idot_well_def_and_rep}. Taking
expectations gives
\[
0
=
\E_\pi[(Lg)(X)]
=
\sum_{i=1}^N \E_\pi[(L_i g_i)(X)]
+
\sum_{i=1}^N \E_\pi[(L_i h_i)(X)].
\]
The first sum is zero, and therefore
\[
\sum_{i=1}^N \E_\pi[(L_i h_i)(X)]=0.
\]
Using Lemma~\ref{lem:Idot_well_def_and_rep},
\[
(\dot I)_{X_i}=-\E_\pi[(L_i h_i)(X)],
\]
we conclude that
\[
\sum_{i=1}^N(\dot I)_{X_i}=0.\qedhere
\]
\end{proof}

\section{Termwise Bounds under Signed Monotonicity}
\label{sec:structural_balance}

The global tradeoff \eqref{eq:tradeoff_thm} shows that
\[
\sum_{i=1}^N \frac{F_{X_i}-1}{\tau_i}\ge 0.
\]
In this section we isolate a natural structural regime in which one has the
stronger \emph{termwise} conclusion
\[
F_{X_i}\ge 1
\qquad\text{for every }i.
\]
The hypothesis is most cleanly stated in terms of a partial order obtained by reversing the order on a subset of coordinates.
Structurally balanced signed interaction graphs provide the natural graph-theoretic realization of the hypothesis, but the theorem is stated directly in the order-theoretic language used in the proof.

\subsection{Signed ordering and our main assumption}
\label{sec:signed_ordering}

Fix a signature vector
\[
\sigma=(\sigma_1,\dots,\sigma_N)\in\{-1,1\}^N.
\]
Define a partial order $\preceq_\sigma$ on $S=\Z_{\ge0}^N$ by
\begin{equation}\label{eq:reflected_order}
x\preceq_\sigma y
\quad\Longleftrightarrow\quad
\sigma_i x_i\le \sigma_i y_i
\qquad\text{for all }i=1,\dots,N.
\end{equation}
Thus on coordinates with $\sigma_i=1$ this is the usual order, whereas on
coordinates with $\sigma_i=-1$ the order is reversed.

Our main structural hypothesis for this section is the following signed monotonicity condition.

\begin{assumption}[Signed monotonicity]\label{ass:signed_monotonicity}
There exists $\sigma\in\{-1,1\}^N$ such that for each $i\in\{1,\dots,N\}$,
whenever $x\preceq_\sigma y$, one has
\begin{equation}\label{eq:signed_monotonicity}
\sigma_i\, f_i(x_{-i}) \le \sigma_i\, f_i(y_{-i}).
\end{equation}
\end{assumption}

Equivalently, for each fixed $i$, the map $x\mapsto f_i(x_{-i})$ is
$\preceq_\sigma$--increasing when $\sigma_i=1$ and
$\preceq_\sigma$--decreasing when $\sigma_i=-1$.

We now describe a standard graph-theoretic sufficient condition for
Assumption~\ref{ass:signed_monotonicity}. Consider the signed interaction graph
with vertex set
\[
\{X_1,\dots,X_N\},
\]
where the vertex $X_i$ represents the $i$th species (equivalently, the counting process for that species).
For each ordered pair $j\neq i$, draw an edge $X_j\to X_i$ if and only if
$f_i$ depends nontrivially on $x_j$. Assume that whenever such a dependence is
present, the map
\[
x_j \mapsto f_i(x_{-i})
\]
is globally monotone when the remaining coordinates are held fixed. Label the
edge $X_j\to X_i$ by $+$ if this map is nondecreasing and by $-$ if it is
nonincreasing.

A signed interaction graph is called \emph{structurally balanced} if there
exists a partition of the vertex set into two classes,
\[
\{X_1,\dots,X_N\}=V_1\cup V_2,
\]
such that every positive edge has both endpoints in the same class, while every
negative edge has its endpoints in different classes.

\begin{proposition}[Graph criterion for signed monotonicity]
\label{prop:positive_parity_implies_signed_monotonicity}
Assume that for each ordered pair $j\neq i$ such that $f_i$ depends nontrivially on $x_j$, the map $x_j\mapsto f_i(x_{-i})$ is globally either nondecreasing or nonincreasing. Then the following are equivalent:
\begin{enumerate}
\item every cycle in the underlying signed graph, obtained by ignoring edge orientation, has positive parity, that is, contains an even number of negative edges;

    \item the signed interaction graph is structurally balanced;
    \item there exists $\sigma\in\{-1,1\}^N$ such that
    Assumption~\ref{ass:signed_monotonicity} holds.
\end{enumerate}
In particular, every structurally balanced interaction graph yields a
signature $\sigma$ for which \eqref{eq:signed_monotonicity} holds.
\end{proposition}

Here and below, \textit{cycle parity} is understood in the underlying signed graph obtained by forgetting edge orientation.

\begin{proof}
The equivalence of (1) and (2) follows from Harary's balance theorem for signed graphs
\cite{harary1953notion}.   We prove $(2)\Rightarrow(3)$ and $(3)\Rightarrow(2)$.

Assume (2). That is, we assume that the signed interaction graph is structurally balanced and we denote the classes by $V_1$ and $V_2$.  We must now construct a $\sigma\in \{-1,1\}^N$ such that Assumption \ref{ass:signed_monotonicity} holds.  Define $\sigma_i=1$ when $X_i\in V_1$ and $\sigma_i=-1$ when
$X_i\in V_2$. Fix $i$. If $X_j\to X_i$ is a positive edge, then $X_j$ and
$X_i$ lie in the same class, so $\sigma_j=\sigma_i$; if $X_j\to X_i$ is a
negative edge, then they lie in different classes, so $\sigma_j=-\sigma_i$.
Thus, for each $j$ on which $f_i$ depends, the map
$x_j\mapsto \sigma_i f_i(x_{-i})$ is nondecreasing when the remaining
coordinates are held fixed. Since this holds for every relevant $j$, it
follows by varying coordinates one at a time that
\[
x\preceq_\sigma y \implies \sigma_i f_i(x_{-i})\le \sigma_i f_i(y_{-i}),
\]
which is exactly \eqref{eq:signed_monotonicity}.

Conversely, assume (3), that there exists a $\sigma\in\{-1,1\}^N$ such that
Assumption~\ref{ass:signed_monotonicity} holds. We must show that the signed interaction graph is structurally balanced.  Let
\[
V_1:=\{X_i:\sigma_i=1\},\qquad V_2:=\{X_i:\sigma_i=-1\}.
\]
If $X_j\to X_i$ is a positive edge, then $f_i$ is nondecreasing in $x_j$.
Choose $x,y$ differing only in coordinate $j$, with $x\preceq_\sigma y$.
Then Assumption~\ref{ass:signed_monotonicity} is compatible with this positive
dependence only when $\sigma_j=\sigma_i$. Similarly, if $X_j\to X_i$ is a
negative edge, then $f_i$ is nonincreasing in $x_j$, and the same comparison
shows that Assumption~\ref{ass:signed_monotonicity} is possible only when
$\sigma_j=-\sigma_i$. Hence positive edges lie within classes and negative
edges go across classes, so the graph is structurally balanced.
\end{proof}

If the interaction graph is connected, which is the case we consider in this paper, then the choice of $\sigma$ is unique
up to global sign flip.   In particular, replacing $\sigma$ by $-\sigma$
reverses the order $\preceq_\sigma$ but leaves
Assumption~\ref{ass:signed_monotonicity} unchanged.

\begin{example}[A 3-node structurally balanced system]

Consider a system with $N=3$ variables and the following monotone dependencies:

\noindent
\begin{minipage}[t]{0.56\textwidth}
\vspace{0pt}
\begin{itemize}
    \item $x_1 \xrightarrow{-} f_2$: $X_1$ negatively affects $f_2$.
    \item $x_3 \xrightarrow{+} f_2$: $X_3$ positively affects $f_2$.
    \item $x_2 \xrightarrow{+} f_3$: $X_2$ positively affects $f_3$.
    \item $x_3 \xrightarrow{-} f_1$: $X_3$ negatively affects $f_1$.
\end{itemize}
\end{minipage}\hfill
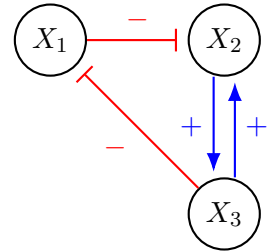
\begin{minipage}[t]{0.38\textwidth}
\vspace{0pt}
\centering
\begin{tikzpicture}[
    node distance={2.3cm},
    main/.style = {draw, circle, thick},
    activation/.style = {-{Latex[length=2.5mm]}, thick, shorten >=2pt},
    inhibition/.style = {-{Bar[width=3mm]}, thick, shorten >=4pt}
]
\node[main] (1) {$X_1$};
\node[main] (2) [right of=1] {$X_2$};
\node[main] (3) [below of=2] {$X_3$};

\draw[inhibition, red] (1) -- node[midway, above] {$-$} (2);
\draw[activation, blue] ([xshift=-4pt]2.south) -- node[midway, left] {$+$} ([xshift=-4pt]3.north);
\draw[activation, blue] ([xshift=4pt]3.north) -- node[midway, right] {$+$} ([xshift=4pt]2.south);
\draw[inhibition, red] (3) -- node[midway, below left] {$-$} (1);
\end{tikzpicture}
\end{minipage}

\vspace{.1in}

By choosing the partition $V_1=\{X_1\}$ and $V_2=\{X_2,X_3\}$, we obtain the
signature vector $\sigma=(1,-1,-1)$. This configuration is valid because the
only negative interactions occur across classes (between $X_1$ and $\{X_2,X_3\}$),
while the positive interactions between $X_2$ and $X_3$ occur within the class $V_2$.

To illustrate \eqref{eq:signed_monotonicity}, consider $f_2$. Since
$\sigma=(1,-1,-1)$ and $\sigma_2=-1$, we must show that
\[
f_2(y_1,y_3)\le f_2(x_1,x_3)\tag{since $\sigma_2 = -1$}
\]
whenever $x_1\le y_1$ (since $\sigma_1=1$) and $y_3\le x_3$ (since $\sigma_3=-1$).
This follows immediately because $f_2$ is nonincreasing in $x_1$ and
nondecreasing in $x_3$, as indicated by the signs in the interaction graph above.\hfill $\triangle$
\end{example}

We next record a canonical example in the opposite direction: the repressilator,
introduced by Elowitz and Leibler \cite{elowitz2000synthetic}, whose interaction
graph is not structurally balanced and therefore does not satisfy the hypothesis
of Theorem~\ref{thm:local_bound_signed_monotonicity} below.

\begin{example}[The repressilator]
Consider the $3$-species repressilator, in which each species represses the next:

\noindent
\begin{minipage}[t]{0.56\textwidth}
\vspace{0pt}
\begin{itemize}
    \item $x_1 \xrightarrow{-} f_2$,
    \item $x_2 \xrightarrow{-} f_3$,
    \item $x_3 \xrightarrow{-} f_1$.
\end{itemize}
\end{minipage}\hfill
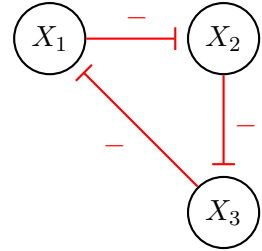
\begin{minipage}[t]{0.38\textwidth}
\vspace{0pt}
\centering
\begin{tikzpicture}[
    node distance={2.3cm},
    main/.style = {draw, circle, thick},
    inhibition/.style = {-{Bar[width=3mm]}, thick, shorten >=4pt}
]
\node[main] (1) {$X_1$};
\node[main] (2) [right of=1] {$X_2$};
\node[main] (3) [below of=2] {$X_3$};

\draw[inhibition, red] (1) -- node[midway, above] {$-$} (2);
\draw[inhibition, red] (2) -- node[midway, right] {$-$} (3);
\draw[inhibition, red] (3) -- node[midway, below left] {$-$} (1);
\end{tikzpicture}
\end{minipage}

\vspace{.1in}

Thus the signed interaction graph is a directed $3$-cycle with three negative
edges. In particular, the unique cycle has negative parity, since it contains
an odd number of negative edges. By
Proposition~\ref{prop:positive_parity_implies_signed_monotonicity}, the graph is
not structurally balanced, and there is no signature
$\sigma\in\{-1,1\}^N$ for which Assumption~\ref{ass:signed_monotonicity} holds.

Accordingly, the termwise conclusion of
Theorem~\ref{thm:local_bound_signed_monotonicity} below is not available from
the signed-monotonicity argument in this case. \hfill $\triangle$
\end{example}

We end this subsection by briefly connecting our setup to the theory of
\emph{monotone systems}. For a signature vector
$\sigma\in\{-1,1\}^N$, monotone-systems theory studies the orthant cone
\[
K_\sigma:=\{z\in\R^N:\sigma_i z_i\ge 0\text{ for all }i\},
\]
and the induced order
\[
x\preceq_\sigma y
\quad\Longleftrightarrow\quad
y-x\in K_\sigma
\quad\Longleftrightarrow\quad
\sigma_i x_i\le \sigma_i y_i
\text{ for all }i.
\]
Thus, for example, if $\sigma=(1,-1,-1)$, then
\[
x\preceq_\sigma y
\quad\Longleftrightarrow\quad
x_1\le y_1,\qquad x_2\ge y_2,\qquad x_3\ge y_3.
\]

In this language, Assumption~\ref{ass:signed_monotonicity} says precisely that,
for each $i$, the signed birth-rate function $\sigma_i f_i$ is increasing with
respect to the order $\preceq_\sigma$. This places us in the same
sign-structured framework used in the monotone-systems literature; see, for
example, Smith's general treatment of monotone dynamical systems and the
biochemical-network formulations of Angeli, De Leenheer, and Sontag
\cite{smith1995monotone,angeli2007monotone,de2007monotone}. For the present
unit-birth model, however, the point is not to invoke monotone-systems results
directly, but to adapt the same signed-order viewpoint to a stochastic setting.
Here the signature $\sigma$ identifies an order under which the dynamics are
order-preserving and the stationary law can be shown to be associated. This
stationary association is the key new step that yields the termwise bound
$F_{X_i}\ge 1$.

This perspective is also complementary to stochastic comparison results for
reaction networks, such as those of Campos et al.~\cite{campos2023comparison}.
There as well, sign and ordering information are used to obtain comparison
results for the dynamics, much as in the monotone-systems literature. Our goal
here is different: we use the signature $\sigma$ from
Assumption~\ref{ass:signed_monotonicity} to deduce a property of the stationary
distribution itself, namely association with respect to $\preceq_\sigma$.

\subsection{Main theorem, and proof outline,  for networks with signed monotonicity}

Proposition~\ref{prop:positive_parity_implies_signed_monotonicity} identifies a broad graph-theoretic regime in which Assumption~\ref{ass:signed_monotonicity} holds. We now show that Assumption~\ref{ass:signed_monotonicity} itself is sufficient to strengthen the global tradeoff to a termwise bound on each component.

\begin{theorem}[Termwise bound under signed monotonicity]
\label{thm:local_bound_signed_monotonicity}
Assume Assumption~\ref{ass:standing} and
Assumption~\ref{ass:signed_monotonicity}. Then, for every $i\in\{1,\dots,N\}$, we have
\[
\Cov_\pi\!\big(f_i(X_{-i}),X_i\big)\ge 0.
\]
In particular, for every $i\in\{1,\dots,N\}$,
\begin{equation}\label{eq:termwise_local_bound_signed}
F_{X_i}\ge 1.
\end{equation}
\end{theorem}

We make two observations about the implications of
Theorem~\ref{thm:local_bound_signed_monotonicity}. First, under
Assumption~\ref{ass:signed_monotonicity}, the global tradeoff
\[
\sum_{i=1}^N \frac{F_{X_i}-1}{\tau_i}\ge 0
\]
upgrades to the termwise conclusion $F_{X_i}\ge 1$ for every $i$. In
particular, no individual component can suppress its Fano factor below the
Poisson baseline in this signed-monotone regime. 
%The examples in Section~\ref{sec:signed_ordering} show that this hypothesis is genuinely restrictive: structurally balanced interaction graphs satisfy it, whereas the repressilator does not.

Second, we will call a signed interaction graph \emph{frustrated} if it is not
structurally balanced, equivalently, if the underlying signed graph contains a
cycle with negative parity. With this terminology, Theorem~\ref{thm:local_bound_signed_monotonicity}
implies that, within the class of networks whose interactions admit a global
sign description, frustration is a necessary condition for sub-Poissonian noise
in at least one component. In other words, structurally balanced networks
cannot achieve $F_{X_i}<1$, so any such noise suppression must come from a
frustrated interaction structure.

\subsubsection{Proof preparation and outline}

To prepare for the proof, let $(P_t)_{t\ge0}$ denote the Markov semigroup of
the chain:
\[
(P_t\phi)(x):=\E_x[\phi(X(t))],
\qquad x\in S,
\]
for bounded measurable $\phi:S\to\R$.

We will also use the following standard notion of positive dependence, adapted
here to the signed order $\preceq_\sigma$.

\begin{definition}[Association with respect to $\preceq_\sigma$]
A probability measure $\mu$ on $S$ is said to be \emph{associated} with respect
to $\preceq_\sigma$ if for all bounded $\preceq_\sigma$--increasing functions
$u,v:S\to\R$,
\[
\Cov_\mu(u,v)\ge 0.
\]
\end{definition}

This is the measure-theoretic version of the notion of association on a
partially ordered space studied by Lindqvist
\cite{lindqvist1988association}. It is also closely related to the classical
notion introduced by Esary, Proschan, and Walkup
\cite{esary1967association}, who define a random vector
$X=(X_1,\dots,X_n)$ to be associated if
\[
\Cov(f(X),g(X))\ge 0
\]
for all coordinatewise nondecreasing functions $f$ and $g$ for which the
covariance exists. This is equivalent to saying that the law $\mu$ of $X$ has
the property that
\[
\Cov_\mu(f,g)\ge 0
\]
for all coordinatewise nondecreasing test functions $f$ and $g$ for which the
covariance exists. In the present setting, we work with the same positive
dependence idea, but formulated directly for a probability measure on the
partially ordered state space $(S,\preceq_\sigma)$ and restricted to bounded
$\preceq_\sigma$--increasing test functions. This bounded formulation is
sufficient for the semigroup and truncation arguments used below.

The key point is that, under Assumption~\ref{ass:signed_monotonicity}, the
signed order $\preceq_\sigma$ is preserved by the dynamics strongly enough to
force association of the stationary law. Once this is established, the desired
covariance bound follows by applying association to the observables $X_i$ and
$f_i(X_{-i})$ (or their negatives when $\sigma_i=-1$).

The proof occupies the remainder of the section and proceeds in four steps:
\begin{enumerate}
\item We show that the semigroup $(P_t)_{t\ge0}$ is monotone with respect
to $\preceq_\sigma$, meaning that $P_t\phi$ is
$\preceq_\sigma$--increasing whenever $\phi$ is bounded and
$\preceq_\sigma$--increasing. We also prove the analogous monotonicity
statement for a finite-box truncation of the chain, which will be used in
Step 2.

    \item We prove that for every $x\in S$ and every $t\ge0$, the transition
    law $P_t(x,\cdot)$ is associated with respect to $\preceq_\sigma$.

\item We use stationarity and ergodicity to deduce that the stationary law
$\pi$ is itself associated with respect to $\preceq_\sigma$.

    \item We apply this association property to obtain
    \[
    \Cov_\pi\!\big(f_i(X_{-i}),X_i\big)\ge 0,
    \]
    and then conclude \eqref{eq:termwise_local_bound_signed} from the
    Fano--covariance identity established earlier.
\end{enumerate}

\subsection{Proof of Theorem~\ref{thm:local_bound_signed_monotonicity}}

We now prove Theorem~\ref{thm:local_bound_signed_monotonicity}. The proof is
organized according to the four steps outlined above. Step 1 establishes the
order-preserving property of the dynamics at the level of the semigroup. Step 2
uses this monotonicity, together with a finite-box approximation, to prove
association of the transition laws $P_t(x,\cdot)$. Step 3 passes from
association of transition laws to association of the stationary distribution
$\pi$. Step 4 applies stationary association to the observables
$f_i(X_{-i})$ and $X_i$, yielding the covariance bound in the theorem and hence
the termwise Fano bound.

\subsubsection{Step 1: Monotonicity of the semigroup}

\begin{proposition}[Monotonicity of the semigroup $P_t$]
\label{prop:monotone_semigroup}
Assume Assumption~\ref{ass:standing} and
Assumption~\ref{ass:signed_monotonicity}. Then, for every $t\ge0$, the
semigroup $P_t$ is monotone with respect to $\preceq_\sigma$: if
$\phi:S\to\R$ is bounded and $\preceq_\sigma$--increasing, then $P_t\phi$ is
also $\preceq_\sigma$--increasing.
\end{proposition}
\begin{proof}
Fix $x,y\in S$ with $x\preceq_\sigma y$. We first explain what we need to
construct. If there is a coupling of the chains started from $x$ and $y$ such
that
\[
X^x(t)\preceq_\sigma X^y(t)
\qquad\text{for all }t\ge 0
\]
almost surely, then for every bounded $\preceq_\sigma$--increasing function
$\phi$,
\[
\phi(X^x(t))\le \phi(X^y(t))
\qquad\text{almost surely}.
\]
Taking expectations gives
\[
(P_t\phi)(x)\le (P_t\phi)(y),
\]
and hence $P_t\phi$ is $\preceq_\sigma$--increasing. Thus it remains only to
construct such an order-preserving coupling.

We use the standard split coupling, written in random-time-change form \cite{anderson2012efficient}.   For each $i$, define
\begin{align*}
b_i^0(z,w)&:= \min\{ f_i(z_{-i}), f_i(w_{-i})\},\\
b_i^x(z,w)&:=f_i(z_{-i})- \min\{ f_i(z_{-i}), f_i(w_{-i})\},\\
b_i^y(z,w)&:=f_i(w_{-i})- \min\{ f_i(z_{-i}), f_i(w_{-i})\},
\end{align*}
and
\begin{align*}
d_i^0(z,w)&:=\min\{z_i/\tau_i, w_i/\tau_i\},\\
d_i^x(z,w)&:=\frac{z_i}{\tau_i} - \min\{z_i/\tau_i, w_i/\tau_i\},\\
d_i^y(z,w)&:=\frac{w_i}{\tau_i} - \min\{z_i/\tau_i, w_i/\tau_i\}.
\end{align*}
 Define $(X^x(t),X^y(t))$ by
\begin{align*}
X^x(t)
&=x+\sum_i
Y_i^{b,0}\!\left(\int_0^t b_i^0(X^x(s),X^y(s))\,ds\right)e_i
+\sum_i
Y_i^{b,x}\!\left(\int_0^t b_i^x(X^x(s),X^y(s))\,ds\right)e_i\\
&\quad-\sum_i
Y_i^{d,0}\!\left(\int_0^t d_i^0(X^x(s),X^y(s))\,ds\right)e_i
-\sum_i
Y_i^{d,x}\!\left(\int_0^t d_i^x(X^x(s),X^y(s))\,ds\right)e_i,
\\[1ex]
X^y(t)
&=y+\sum_i
Y_i^{b,0}\!\left(\int_0^t b_i^0(X^x(s),X^y(s))\,ds\right)e_i
+\sum_i
Y_i^{b,y}\!\left(\int_0^t b_i^y(X^x(s),X^y(s))\,ds\right)e_i\\
&\quad-\sum_i
Y_i^{d,0}\!\left(\int_0^t d_i^0(X^x(s),X^y(s))\,ds\right)e_i
-\sum_i
Y_i^{d,y}\!\left(\int_0^t d_i^y(X^x(s),X^y(s))\,ds\right)e_i,
\end{align*}
where the $e_i$ are the standard unit vectors in $\mathbb R^N$ and the $Y_i^*$ are independent unit-rate Poisson processes. Since
\[
b_i^0+b_i^x=f_i(z_{-i}),\qquad b_i^0+b_i^y=f_i(w_{-i}),
\]
and
\[
d_i^0+d_i^x=\frac{z_i}{\tau_i},\qquad
d_i^0+d_i^y=\frac{w_i}{\tau_i},
\]
the two marginals have the correct transition rates.

It remains to check that the coupling preserves the signed order. By
nonexplosion, it suffices to check this at a single jump time, which we denote
by $t$. Suppose that
\[
X^x(t-)\preceq_\sigma X^y(t-)
\]
just before the jump, and consider a jump in coordinate $i$.  We consider the various cases.

First, common births and common deaths preserve the order. Indeed, a common
birth adds $e_i$ to both chains, while a common death subtracts $e_i$ from both
chains; in either case the signed comparison between the two states is
unchanged.

Second, consider unmatched deaths. If $\sigma_i=1$, then
$X_i^x(t-)\le X_i^y(t-)$, so an unmatched death can occur only in the
$y$-chain. Since the state space is integer-valued, this moves the two
coordinates closer together without reversing the inequality. If
$\sigma_i=-1$, then $X_i^x(t-)\ge X_i^y(t-)$, so an unmatched death can occur
only in the $x$-chain, and again the two coordinates move closer together
without reversing the signed order.

Third, consider unmatched births when the two $i$th coordinates are not equal. If
$\sigma_i=1$, then $X_i^x(t-)<X_i^y(t-)$. An unmatched birth in the $x$-chain
gives
\[
X_i^x(t-)+1\le X_i^y(t-),
\]
while an unmatched birth in the $y$-chain only increases the upper coordinate.
If $\sigma_i=-1$, then $X_i^x(t-)>X_i^y(t-)$. An unmatched birth in the
$x$-chain only increases the upper coordinate in the signed order, while an
unmatched birth in the $y$-chain gives
\[
X_i^y(t-)+1\le X_i^x(t-).
\]
Thus unmatched births also preserve the signed order whenever the two
coordinates are not equal.

It remains only to consider unmatched births when
\[
X_i^x(t-)=X_i^y(t-).
\]
Here Assumption~\ref{ass:signed_monotonicity} rules out exactly the
order-violating unmatched birth. If $\sigma_i=1$, then
\[
f_i(X^x_{-i}(t-))\le f_i(X^y_{-i}(t-)),
\]
so $b_i^x(X^x(t-),X^y(t-))=0$ and an unmatched birth in the $x$-chain is
impossible. If $\sigma_i=-1$, then
\[
f_i(X^x_{-i}(t-))\ge f_i(X^y_{-i}(t-)),
\]
so $b_i^y(X^x(t-),X^y(t-))=0$ and an unmatched birth in the $y$-chain is
impossible. In both cases, the only possible unmatched birth is
order-preserving.

Therefore every jump of the coupled process preserves
$X^x\preceq_\sigma X^y$. Since the order holds at time $0$, it holds for all
$t\ge0$ almost surely. As observed at the start of the proof, this implies that
$P_t$ maps bounded $\preceq_\sigma$--increasing functions to bounded
$\preceq_\sigma$--increasing functions.
\end{proof}

For use in the next step, we also introduce the finite-box version of the
process. For $M\ge 1$, let
\[
\Lambda_M:=\{0,1,\dots,M\}^N,
\]
and let $L^{(M)}$ be the generator on $\Lambda_M$ obtained by suppressing births
at the boundary:
\begin{equation}\label{eq:finite_box_generator}
\begin{aligned}
L^{(M)}\phi(z)
&=
\sum_{i=1}^N f_i(z_{-i})\mathbf 1_{\{z_i<M\}}
\big(\phi(z+e_i)-\phi(z)\big)  + \sum_{i=1}^N \frac{z_i}{\tau_i}
\big(\phi(z-e_i)-\phi(z)\big),
\end{aligned}
\end{equation}
with the convention that the death term is zero when $z_i=0$. Let
$(P_t^{(M)})_{t\ge0}$ denote the corresponding finite-state semigroup.

\begin{corollary}[Finite-box monotonicity]
\label{cor:finite_box_monotonicity}
Assume Assumption~\ref{ass:standing} and
Assumption~\ref{ass:signed_monotonicity}. Then, for every $M\ge1$ and
every $t\ge0$, the semigroup $P_t^{(M)}$ is monotone with respect to
$\preceq_\sigma$: if $\phi:\Lambda_M\to\R$ is
$\preceq_\sigma$--increasing, then $P_t^{(M)}\phi$ is
$\preceq_\sigma$--increasing.
\end{corollary}

\begin{proof}
The proof is the same split-coupling argument as above, with the birth rates
replaced by
\[
f_i(z_{-i})\mathbf 1_{\{z_i<M\}}.
\]
The boundary factor causes no new difficulty. If the two $i$th coordinates are
equal, then the boundary indicators are equal, so the equal-coordinate case is
exactly the same as for the full chain. If the two $i$th coordinates are
unequal, then any allowed unmatched birth is order-preserving for the same
integer-valued reason as above, while suppressing a boundary birth only removes
a possible jump. Thus the finite-box coupling also preserves
$\preceq_\sigma$, and the monotonicity of $P_t^{(M)}$ follows.
\end{proof}

\subsubsection{Step 2: Association of the transition laws}

We next prove the association statement for the transition laws.

\begin{proposition}[Association of the transition laws]
\label{prop:transition_law_associated_signed}
Assume Assumption~\ref{ass:standing} and
Assumption~\ref{ass:signed_monotonicity}. Then, for every $x\in S$ and every
$t\ge 0$, the transition law $P_t(x,\cdot)$ is associated with respect to
$\preceq_\sigma$.
\end{proposition}

\begin{proof}
Fix $x\in S$, $t\ge0$, and bounded $\preceq_\sigma$--increasing functions
$u,v:S\to\R$. To prove that $P_t(x,\cdot)$ is associated with respect to
$\preceq_\sigma$, it suffices to show that
\[
P_t(uv)(x)
\ge
(P_tu)(x)(P_tv)(x),
\]
since this inequality is exactly
\[
\Cov_{P_t(x,\cdot)}(u,v)\ge0.
\]

We first prove a finite-box version of this inequality. Fix $M$ large enough
that $x\in\Lambda_M$. We work with the finite-state semigroup
$(P_t^{(M)})_{t\ge0}$ generated by $L^{(M)}$ in
\eqref{eq:finite_box_generator}. Throughout this finite-box argument, when
$P_t^{(M)}$ is applied to a function originally defined on $S$, we mean its
restriction to $\Lambda_M$.

For functions $a,b:\Lambda_M\to\R$, define
\[
\Gamma_M(a,b)
:=
L^{(M)}(ab)-aL^{(M)}b-bL^{(M)}a.
\]
Writing
\[
L^{(M)}h(z)=\sum_{w\neq z} q_M(z,w)\big(h(w)-h(z)\big),
\]
where $q_M(z,w)$ are the transition rates of the finite-box chain, we obtain
\begin{align*}
\Gamma_M(a,b)(z)
&=
\big(L^{(M)}(ab)\big)(z)
-a(z)\big(L^{(M)}b\big)(z)
-b(z)\big(L^{(M)}a\big)(z)\\
&=
\sum_{w\neq z} q_M(z,w)
\Big[
a(w)b(w)-a(z)b(z)  \\
&\hspace{3.5cm}
-a(z)\big(b(w)-b(z)\big)
-b(z)\big(a(w)-a(z)\big)
\Big]  \\
&=
\sum_{w\neq z} q_M(z,w)
\big(a(w)-a(z)\big)\big(b(w)-b(z)\big).
\end{align*}

For $0\le s \le t$, set
\[
a_s:=P_{t-s}^{(M)}u,
\qquad
b_s:=P_{t-s}^{(M)}v,
\qquad
F_s:=a_s b_s.
\]
Note that $a_t = u$ and $b_t = v$.  Define
\[
H(s):=P_s^{(M)}F_s(x)
=
P_s^{(M)}
\Big[
(P_{t-s}^{(M)}u)(P_{t-s}^{(M)}v)
\Big](x).
\]
Since $\Lambda_M$ is finite, the forward and backward equations hold without
domain issues. In particular, for any differentiable family
$F_s:\Lambda_M\to\R$,
\[
\frac{d}{ds}P_s^{(M)}F_s
=
P_s^{(M)}L^{(M)}F_s+P_s^{(M)}\dot F_s.
\]
Also,
\[
\dot a_s=-L^{(M)}a_s,
\qquad
\dot b_s=-L^{(M)}b_s,
\]
and hence
\[
\dot F_s=-(L^{(M)}a_s)b_s-a_s(L^{(M)}b_s).
\]
Applying the preceding differentiation identity to $H(s)=P_s^{(M)}F_s(x)$ gives
\begin{align*}
H'(s)
&=
P_s^{(M)}
\Big[
L^{(M)}(a_s b_s)
-(L^{(M)}a_s)b_s
-a_s(L^{(M)}b_s)
\Big](x)\\
&=
P_s^{(M)}\Gamma_M(a_s,b_s)(x).
\end{align*}

By Corollary~\ref{cor:finite_box_monotonicity}, $P_r^{(M)}$ is monotone with
respect to $\preceq_\sigma$ for every $r\ge0$. Therefore $a_s$ and $b_s$ are
$\preceq_\sigma$--increasing. If $q_M(z,w)>0$, then $w$ is obtained from $z$
by changing a single coordinate by $+1$ or $-1$, so $z$ and $w$ are comparable
under $\preceq_\sigma$. It follows that
\[
\big(a_s(w)-a_s(z)\big)\big(b_s(w)-b_s(z)\big)\ge 0
\]
for every allowed jump $z\to w$. Hence
\[
\Gamma_M(a_s,b_s)(z)\ge0
\qquad\text{for every }z\in\Lambda_M.
\]
Since $P_s^{(M)}$ preserves nonnegative functions, we have $H'(s)\ge0$ for
all $s\in[0,t]$. Thus $H(t)\ge H(0)$, that is,
\[
P_t^{(M)}(uv)(x)
\ge
(P_t^{(M)}u)(x)(P_t^{(M)}v)(x),
\]
where we utilized that  $H(t)=P_t^{(M)}(uv)(x)$  while
$H(0)=(P_t^{(M)}u)(x)(P_t^{(M)}v)(x)$.

We now pass to the full chain. Let $\tau_M$ be the first time the full chain,
started from $x$, exits $\Lambda_M$. Couple the full chain, denoted $X$, and the finite-box
chain, denoted $X^{(M)}$, also started from $x$, using the same driving processes up to time
$\tau_M$. On the event $\{\tau_M>t\}$, the two chains agree at time $t$.

By nonexplosion, the full chain makes only finitely many jumps on $[0,t]$
almost surely. Hence, for each fixed sample path, $\tau_M>t$ for all sufficiently
large $M$, and therefore
\[
\mathbb P_x(\tau_M\le t)\to 0
\qquad\text{as }M\to\infty.
\]
Therefore, under this coupling, $X^{(M)}(t)\to X(t)$ almost surely as $M\to \infty$. Since $u$, $v$, and
$uv$ are bounded, bounded convergence gives
\[
P_t^{(M)}u(x)\to P_tu(x),
\qquad
P_t^{(M)}v(x)\to P_tv(x),
\qquad
P_t^{(M)}(uv)(x)\to P_t(uv)(x),
\]
as $M\to \infty$.  Hence, letting $M\to\infty$ in the finite-box inequality gives
\[
P_t(uv)(x)
\ge
(P_tu)(x)(P_tv)(x).
\]
As noted at the start of the proof, this is precisely association of
$P_t(x,\cdot)$ with respect to $\preceq_\sigma$.
\end{proof}

\subsubsection{Step 3: Association of the stationary law}

We now pass from association of the transition laws to association of the
stationary distribution.

\begin{proposition}[Association of the stationary law]
\label{prop:stationary_law_associated_signed}
Assume Assumption~\ref{ass:standing} and
Assumption~\ref{ass:signed_monotonicity}. Then the stationary law $\pi$ is
associated with respect to $\preceq_\sigma$.
\end{proposition}

\begin{proof}
Let $u,v:S\to\R$ be bounded and $\preceq_\sigma$--increasing. By
Proposition~\ref{prop:transition_law_associated_signed}, for every $x\in S$
and every $t\ge0$,
\[
P_t(uv)(x)\ge (P_tu)(x)(P_tv)(x).
\]
Integrating with respect to $\pi$ gives
\[
\sum_{x\in S}\pi(x)P_t(uv)(x)
\ge
\sum_{x\in S}\pi(x)(P_tu)(x)(P_tv)(x).
\]
By stationarity of $\pi$, the left-hand side is
\[
\sum_{x\in S}\pi(x)P_t(uv)(x)=\E_\pi[u(X)v(X)],
\]
where $X\sim \pi$.
Thus
\[
\E_\pi[u(X)v(X)]
\ge
\sum_{x\in S}\pi(x)(P_tu)(x)(P_tv)(x).
\]

Since the chain is irreducible and positive recurrent under the standing
hypotheses, the ergodic theorem for countable-state continuous-time Markov
chains gives
\[
P_tg(x)\to \E_\pi[g(X)]
\qquad\text{as }t\to\infty
\]
for every bounded $g:S\to\R$.
 Applying this to $u$ and $v$, and using bounded convergence
with respect to $\pi$, we obtain
\[
\sum_{x\in S}\pi(x)(P_tu)(x)(P_tv)(x)
\to
\E_\pi[u(X)]\,\E_\pi[v(X)].
\]
Letting $t\to\infty$ therefore yields
\[
\E_\pi[u(X)v(X)]\ge \E_\pi[u(X)]\,\E_\pi[v(X)].
\]
Since this holds for all bounded $\preceq_\sigma$--increasing $u$ and $v$,
$\pi$ is associated with respect to $\preceq_\sigma$.
\end{proof}

\subsubsection{Step 4: The covariance bound and conclusion}

We now finish the proof of Theorem~\ref{thm:local_bound_signed_monotonicity}.

\begin{proof}[Proof of Theorem~\ref{thm:local_bound_signed_monotonicity}]
Fix $i\in\{1,\dots,N\}$. Let $X\sim\pi$. We must show that
\[
\Cov_\pi\!\big(f_i(X_{-i}),X_i\big)\ge 0.
\]

For $M\ge1$, define the bounded truncations
\[
U_M(x):=f_i(x_{-i})\wedge M,
\qquad
V_M(x):=x_i\wedge M.
\]
We claim that
\begin{equation}\label{eq:truncated_cov_nonnegative_signed}
\Cov_\pi\!\big(U_M(X),V_M(X)\big)\ge0
\qquad\text{for every }M\ge1.
\end{equation}

Suppose first that $\sigma_i=1$. Then the coordinate map $x\mapsto x_i$ is
$\preceq_\sigma$--increasing, and Assumption~\ref{ass:signed_monotonicity}
implies that $x\mapsto f_i(x_{-i})$ is also $\preceq_\sigma$--increasing.
Since $z\mapsto z\wedge M$ is nondecreasing, both $U_M$ and $V_M$ are bounded
$\preceq_\sigma$--increasing functions. Hence
\eqref{eq:truncated_cov_nonnegative_signed} follows from
Proposition~\ref{prop:stationary_law_associated_signed}.

Suppose now that $\sigma_i=-1$. Then both $x\mapsto x_i$ and
$x\mapsto f_i(x_{-i})$ are $\preceq_\sigma$--decreasing. Therefore
$-U_M$ and $-V_M$ are bounded $\preceq_\sigma$--increasing functions. Applying
Proposition~\ref{prop:stationary_law_associated_signed} gives
\[
\Cov_\pi\!\big(-U_M(X),-V_M(X)\big)\ge0.
\]
Since
\[
\Cov_\pi\!\big(-U_M(X),-V_M(X)\big)
=
\Cov_\pi\!\big(U_M(X),V_M(X)\big),
\]
\eqref{eq:truncated_cov_nonnegative_signed} holds in this case as well.

It remains to let $M\to\infty$. By Assumption~\ref{ass:standing}, in particular
\Aone and \Athree, we have
\[
\E_\pi[f_i(X_{-i})^2]<\infty,
\qquad
\E_\pi[X_i^2]<\infty.
\]
Here we use $f_i(x_{-i})\le \Lambda(x)$ and the at-most-linear growth bound
on $\Lambda$.
Thus $f_i(X_{-i})X_i$ is integrable by Cauchy--Schwarz. Since
\[
U_M(X)\nearrow f_i(X_{-i}),
\qquad
V_M(X)\nearrow X_i,
\qquad
U_M(X)V_M(X)\nearrow f_i(X_{-i})X_i,
\]
the monotone convergence theorem gives
\[
\E_\pi[U_M(X)]\to \E_\pi[f_i(X_{-i})],
\qquad
\E_\pi[V_M(X)]\to \E_\pi[X_i],
\]
and
\[
\E_\pi[U_M(X)V_M(X)]
\to
\E_\pi[f_i(X_{-i})X_i].
\]
Therefore, letting $M\to\infty$ in
\eqref{eq:truncated_cov_nonnegative_signed}, we obtain
\[
\Cov_\pi\!\big(f_i(X_{-i}),X_i\big)\ge0.
\]

Finally, by the Fano--covariance identity \eqref{eq:fano_cov_identity_thm},
\[
F_{X_i}
=
1+
\frac{\Cov_\pi\!\big(f_i(X_{-i}),X_i\big)}{\mu_i}.
\]
Since $\mu_i\in(0,\infty)$ by Assumption~\ref{ass:standing}, the covariance
bound implies
\[
F_{X_i}\ge1.
\]
This proves the theorem.
\end{proof}

\bibliographystyle{plain}
\bibliography{references_Paulsson}

\end{document}